%% file: Li_Liu_Yamamoto_ArXiV_140625.tex
\documentclass[a4paper]{article}

\def\shorttitle{Multi-Term Time-Fractional Diffusion Equations}
\def\shortauthor{Z. Li, Y. Liu and M. Yamamoto}

\usepackage{amsfonts}
\usepackage{amsmath}
\usepackage{amssymb}
\usepackage{amscd}
\usepackage{amsbsy}
\usepackage{amsthm}
\usepackage{bm}
\usepackage{empheq}
\usepackage{graphicx}
\usepackage{psfrag}
\usepackage{color}
\usepackage{cite}
\usepackage{multicol}

\allowdisplaybreaks[4]

\newfont{\myfnt}{cmssi10 scaled 1440}
\numberwithin{equation}{section}
\catcode`@=11
\def\ps@nk{\def\@oddhead{\vbox{\hbox to \hsize{\pic \footnotesize \it \shorttitle
\hfill \rm \thepage} \vspace{1mm} \vspace*{-2mm}}}
\def\@evenhead{\vbox{\hbox to \hsize{\pic \footnotesize \rm \thepage \hfill \it \shortauthor}
\vspace{1mm} \vspace*{-2mm}}}
\def\@oddfoot{} \def\@evenfoot{}}
\def\ps@first{\def\@oddhead{\vbox{\hbox to \hsize{\pic \footnotesize
} \break}}
\def\@oddfoot{} \def\@evenfoot{}}
\newtheoremstyle{thmstyle}
  {6pt}
  {6pt}
  {\it}
  {}
  {\bf}
  {}
  {.5em}
  {}
\newtheoremstyle{remstyle}
  {6pt}
  {6pt}
  {\rm}
  {}
  {\bf}
  {}
  {.5em}
  {}
\def\Section#1{\Sec{\large #1} \setcounter{equation}{0} \vskip -6mm \indent}
\def\Sec{\@Startsection{section}{1}{\z@}
                                   {-3.5ex \@plus -1ex \@minus -.2ex}%
                                   {2.3ex \@plus.2ex}%
                                   {\normalfont\large\bfseries\boldmath}}
\def\@Startsection#1#2#3#4#5#6{%
  \if@noskipsec \leavevmode \fi
  \par
  \@tempskipa #4\relax
  \@afterindenttrue
  \ifdim \@tempskipa <\z@
    \@tempskipa -\@tempskipa \@afterindentfalse
  \fi
  \if@nobreak
    \everypar{}%
  \else
    \addpenalty\@secpenalty\addvspace\@tempskipa
  \fi
  \@ifstar
    {\@ssect{#3}{#4}{#5}{#6}}%
    {\@dblarg{\@Sect{#1}{#2}{#3}{#4}{#5}{#6}}}}
\def\@Sect#1#2#3#4#5#6[#7]#8{%
  \ifnum #2>\c@secnumdepth
    \let\@svsec\@empty
  \else
    \refstepcounter{#1}%
    \protected@edef\@svsec{\@seccntformat{#1}\relax}%
  \fi
  \@tempskipa #5\relax
  \ifdim \@tempskipa>\z@
    \begingroup
      #6{%
          \@hangfrom{\hskip #3\relax\@svsec \hskip -2.5mm}%
          \interlinepenalty \@M #8\@@par}
    \endgroup
    \csname #1mark\endcsname{#7}%
    \addcontentsline{toc}{#1}{%
      \ifnum #2>\c@secnumdepth \else
        \protect\numberline{\csname the#1\endcsname}%
      \fi
      #7}%
  \else
    \def\@svsechd{%
      #6{\hskip #3\relax
      \@svsec #8}%
      \csname #1mark\endcsname{#7}%
      \addcontentsline{toc}{#1}{%
        \ifnum #2>\c@secnumdepth \else
          \protect\numberline{\csname the#1\endcsname}%
        \fi
        #7}}%
  \fi
  \@xsect{#5}}
\renewenvironment{abstract}{%
        \small
        \quotation
         \noindent {\bfseries \abstractname } }%
      {\if@twocolumn\else\endquotation\fi}
\def\Subsection#1{\Subsec{#1} \vskip -6mm \indent}
\def\Subsec{\@StartSubsection{subsection}{2}{\z@}%
                                     {-3.25ex\@plus -1ex \@minus -.2ex}%
                                     {1.5ex \@plus .2ex}%
                                     {\normalfont\normalsize\bfseries\boldmath}}
\def\@StartSubsection#1#2#3#4#5#6{%
  \if@noskipsec \leavevmode \fi
  \par
  \@tempskipa #4\relax
  \@afterindenttrue
  \ifdim \@tempskipa <\z@
    \@tempskipa -\@tempskipa \@afterindentfalse
  \fi
  \if@nobreak
    \everypar{}%
  \else
    \addpenalty\@secpenalty\addvspace\@tempskipa
  \fi
  \@ifstar
    {\@ssect{#3}{#4}{#5}{#6}}%
    {\@dblarg{\@SubSect{#1}{#2}{#3}{#4}{#5}{#6}}}}
\def\@SubSect#1#2#3#4#5#6[#7]#8{%
  \ifnum #2>\c@secnumdepth
    \let\@svsec\@empty
  \else
    \refstepcounter{#1}%
    \protected@edef\@svsec{\@seccntformat{#1}\relax}%
  \fi
  \@tempskipa #5\relax
  \ifdim \@tempskipa>\z@
    \begingroup
      #6{%
          \@hangfrom{\hskip #3\relax\@svsec\hskip -1.5mm}%
          \interlinepenalty \@M #8\@@par}
    \endgroup
    \csname #1mark\endcsname{#7}%
    \addcontentsline{toc}{#1}{%
      \ifnum #2>\c@secnumdepth \else
        \protect\numberline{\csname the#1\endcsname}%
      \fi
      #7}%
  \else
    \def\@svsechd{%
      #6{\hskip #3\relax
      \@svsec #8}%
      \csname #1mark\endcsname{#7}%
      \addcontentsline{toc}{#1}{%
        \ifnum #2>\c@secnumdepth \else
          \protect\numberline{\csname the#1\endcsname}%
        \fi
        #7}}%
  \fi
  \@xsect{#5}}
\def\list#1#2{\ifnum \@listdepth >5\relax \@toodeep \else \global
\advance \@listdepth\@ne \fi \rightmargin \z@ \listparindent\z@
\itemindent\z@ \csname @list\romannumeral\the\@listdepth\endcsname
\def\@itemlabel{#1}\let\makelabel\@mklab \@nmbrlistfalse #2\relax
\@trivlist \parskip 0pt \parindent\listparindent \advance \linewidth
-\rightmargin \advance\linewidth -\leftmargin \advance\@totalleftmargin
\leftmargin \parshape \@ne \@totalleftmargin \linewidth \ignorespaces}
\renewcommand{\@makecaption}[2]{\begin{center}#1. #2\end{center}}
\catcode`@=12 
\theoremstyle{thmstyle}
\newtheorem{thm}{\indent Theorem}[section]
\newtheorem{lem}[thm]{\indent Lemma}

\theoremstyle{remstyle}
\newtheorem{rem}{\indent \bf Remark}[section]

\newsavebox{\mygraphic}
\def\pic{\begin{picture}(0,0) \put(-210,-1250){\usebox{\mygraphic}} \end{picture}}
\newfont{\HUGEbf}{cmbx10 scaled 3500}
\definecolor{gray}{rgb}{0.9,0.9,0.9}
\def\thebibliography#1{\section*{\bf \large References}
\list{[\arabic{enumi}]} {\settowidth \labelwidth{[#1]} \leftmargin
\labelwidth \advance \leftmargin \labelsep \usecounter{enumi}}
\def\newblock{\hskip .11em plus .33em minus .07em} \footnotesize \sloppy \clubpenalty
4000 \widowpenalty 4000 \sfcode`\.=1000 \relax}

\def\BC{\mathbb C}

\def\BN{\mathbb N}
\def\BR{\mathbb R}

\def\cA{\mathcal A}

\def\cC{\mathcal C}
\def\cD{\mathcal D}

\def\cL{\mathcal L}

\def\cQ{\mathcal Q}

\def\cU{\mathcal U}

\def\rd{\mathrm d}

\def\e{\mathrm e}

\def\ri{\mathrm i}

\def\Ga{\Gamma}

\def\Om{\Omega}
\def\al{\alpha}
\def\be{\beta}
\def\ga{\gamma}
\def\de{\delta}

\def\ve{\varepsilon}
\def\ze{\zeta}
\def\te{\theta}

\def\la{\lambda}

\def\vp{\varphi}

\def\f{\frac}

\def\nb{\nabla}
\def\ov{\overline}
\def\pa{\partial}

\def\wh{\widehat}
\def\wt{\widetilde}

\theoremstyle{definition}

\numberwithin{equation}{section}

\title{\Large\bf\boldmath Initial-Boundary Value Problems for Multi-Term\\
Time-Fractional Diffusion Equations with\\
Positive Constant Coefficients$^*$}

\author{\large Zhiyuan LI$^\dag$\qquad Yikan LIU$^\dag$\qquad Masahiro YAMAMOTO$^\dag$}

\date{}

\begin{document}

\maketitle

\thispagestyle{first}
\renewcommand{\thefootnote}{\fnsymbol{footnote}}

\footnotetext{\hspace*{-5mm} \begin{tabular}{@{}r@{}p{13cm}@{}} &
Manuscript last updated: \today.\\
$^\dag$ & Graduate School of Mathematical Sciences, the University of Tokyo,
3-8-1 Komaba, Meguro-ku, Tokyo 153-8914, Japan.\\
& E-mail: zyli@ms.u-tokyo.ac.jp, ykliu@ms.u-tokyo.ac.jp,
myama@ms.u-tokyo.ac.jp\\
$^*$ & This work was supported by the Program for Leading Graduate Schools,
MEXT, Japan.
\end{tabular}}

\renewcommand{\thefootnote}{\arabic{footnote}}

\begin{abstract}
In this paper, we investigate the well-posedness and the long-time asymptotic
behavior for the initial-boundary value problem for multi-term time-fractional
diffusion equations, where the time differentiation consists of a finite
summation of Caputo derivatives with decreasing orders in $(0,1)$ and positive
constant coefficients. By exploiting several important properties of
multinomial Mittag-Leffler functions, various estimates follow from the
explicit solutions in form of these special functions. Then the uniqueness and
continuous dependency upon initial value and source term are established, from
which the continuous dependence of solution of Lipschitz type with respect to
various coefficients is also verified. Finally, by a Laplace transform
argument, it turns out that the decay rate of the solution as $t\to\infty$ is
dominated by the minimum order of the time-fractional derivatives.

\vskip 4.5mm

\noindent\begin{tabular}{@{}l@{ }p{10cm}} {\bf Keywords } & Initial-boundary
value problem, Time-fractional diffusion equation,\\
& Multinomial Mittag-Leffler function, Well-posedness,\\
& Long-time asymptotic behavior, Laplace transform
\end{tabular}

\end{abstract}

\baselineskip 14.5pt

\setlength{\parindent}{1.5em}

\setcounter{section}{0}

\Section{Introduction}\label{sec-intro}

Let $\Om$ be an open bounded domain in $\BR^d$ with a smooth boundary (for
example, of $C^2$ class) and $T>0$ be fixed arbitrarily. For a fixed positive
integer $m$, let $\al_j$ and $q_j$ ($j=1,\ldots,m$) be positive constants such
that $1>\al_1>\cdots>\al_m>0$. Consider the following initial-boundary value
problem for the multi-term time-fractional diffusion equation
\begin{empheq}[left=\empheqlbrace,right=]{alignat=2}
& \sum_{j=1}^mq_j\pa_t^{\al_j}u(x,t)=Lu(x,t)+F(x,t), & \quad & x\in\Om,\
0<t\le T,\label{eq-gov-u}\\
& u(x,t)=0, & \quad & x\in\pa\Om,\ 0<t\le T,\\
& u(x,0)=a(x), & \quad & x\in\Om,\label{eq-IC-u}
\end{empheq}
where $L$ is a symmetric uniformly elliptic operator with the homogeneous
Dirichlet boundary condition, and we can assume $q_1=1$ without lose of
generality. The regularities of the initial value $a$ and the source term $F$
will be specified later. Here $\pa_t^{\al_j}$ denotes the Caputo derivative
defined by
\[
\pa_t^{\al_j}f(t):=\f1{\Ga(1-\al_j)}\int_0^t\f{f'(s)}{(t-s)^{\al_j}}\,\rd s,
\]
where $\Ga(\,\cdot\,)$ is a usual Gamma function. For various properties of the
Caputo derivative, we refer to Kilbas et al. \cite{KST06} and Podlubny
\cite{P99}. See also \cite{GM97,SKM93} for further contents on fractional
calculus.

In the case of $m=1$, equation \eqref{eq-gov-u} is reduced to its single-term
counterpart
\begin{equation}\label{eq-STTFDE}
\pa_t^\al u=Lu+F\quad\mbox{in }\Om\times(0,T],\ \al\in(0,1).
\end{equation}
The above formulation has been studied extensively from different aspects due
to its vast capability of modeling the anomalous diffusion phenomena in highly
heterogeneous aquifer and complex viscoelastic material (see Adams \& Gelhar
\cite{AG92}, Ginoa et al. \cite{GCR92}, Hatano \& Hatano \cite{HH98},
Nigmatullin \cite{N86} and the references therein). Indeed, although the
single-term time-fractional diffusion equation inherits certain properties from
the diffusion equation with time derivative of natural number order, it differs
considerably from the traditional one especially in sense of its limited
smoothing effect in space and slow decay in time. In Luchko \cite{L09}, a
maximum principle of the initial-boundary value problem for \eqref{eq-STTFDE}
was established, and the uniqueness of a classical solution was proved. Luchko
\cite{L10} represented the generalized solution to \eqref{eq-STTFDE} with $F=0$
by means of the Mittag-Leffler function and gave the unique existence result.
Sakamoto \& Yamamoto \cite{SY11} carried out a comprehensive investigation
including the well-posedness of the initial-boundary value problem for
\eqref{eq-STTFDE} as well as the long-time asymptotic behavior of the solution.
It turns out that the spatial regularity of the solution is only moderately
improved from that of the initial value, and the solution decays with order
$t^{-\al}$ as $t\to\infty$. Recently, the Lipschitz stability of the solution
to \eqref{eq-STTFDE} with respect to $\al$ and the diffusion coefficient was
proved as a byproduct of an inverse coefficient problem in Li et al.
\cite{LZJY13}. For other discussions concerning equation \eqref{eq-STTFDE}, see
e.g., Gorenflo et al. \cite{GLZ99} and Luchko \cite{L08}, Pr\"uss \cite{P93}.
Regarding numerical treatments, we refer to Liu et al. \cite{LZATB07} and
Meerschaert \& Tadjeran \cite{MT04} for the finite difference method and Jin et
al. \cite{JLZ13} for the finite element method.

As natural extension, equation \eqref{eq-gov-u} is expected to improve the
modeling accuracy in depicting the anomalous diffusion due to its potential
feasibility. However, to the authors' best knowledge, published works on this
extension are quite limited in spite of rich literatures on its single-term
version. Luchko \cite{L11} developed the maximum principle for problem
\eqref{eq-gov-u}--\eqref{eq-IC-u} and constructed a generalized solution when
$F=0$ by means of the multinomial Mittag-Leffler functions. Jiang et al.
\cite{JLTB12} considered fractional derivatives in both time and space and
derived analytical solutions. As for the asymptotic behavior, for $m=2$ it
reveals in Mainardi et al. \cite{MMPG08} that the dominated decay rate of the
solution is related to the minimum order of time fractional derivative. On the
other hand, Beckers \& Yamamoto \cite{BY13} investigated
\eqref{eq-gov-u}--\eqref{eq-IC-u} in a slightly more general formulation and
obtained a weaker regularity result than that in \cite{SY11}.

In this paper, we are concerned with the well-posedness and the long-time
asymptotic behavior of the solution to the initial-boundary value problem
\eqref{eq-gov-u}--\eqref{eq-IC-u}, and we attempt to establish results parallel
to that for the single-term case. On basis of the explicit representation of
the solution, we give estimates for the solution by exploiting several
properties of the multinomial Mittag-Leffler function. Moreover, as long as the
continuous dependency of the solution on the initial value and the source term
is verified, we can also deduce the Lipschitz stability of the solution to
\eqref{eq-gov-u}--\eqref{eq-IC-u} with respect to $\al_j$, $q_j$
($j=1,\ldots,m$) and the diffusion coefficient immediately. For the long-time
asymptotic behavior, we employ the Laplace transform in time to show that the
decay rate as $t\to\infty$ is indeed $t^{-\al_m}$, where $\al_m$ is the minimum
order of Caputo derivatives in time.

The rest of this paper is organized as follows. The main results concerning
problem \eqref{eq-gov-u}--\eqref{eq-IC-u} are collected in Section
\ref{sec-main}, which includes theorems on well-posedness and long-time
asymptotic behavior of the solution. The proofs of the main theorems are
postponed to Section \ref{sec-proof}, which is further divided into three
subsections. Subsection \ref{sec-multiML} is devoted to a close scrutiny of the
multinomial Mittag-Leffler functions, which serves as essential keys in the
proofs of well-posedness results in Subsection \ref{sec-stability}. Due to the
difference of techniques, the asymptotic behavior is proved independently in
Subsection \ref{sec-asymp}. Finally, concluding remarks are given in Section
\ref{sec-rem}.

\Section{Main Results}\label{sec-main}

In this section, we state the main results obtained in this paper. More
precisely, we give a priori estimates for the solution $u$ to
\eqref{eq-gov-u}--\eqref{eq-IC-u} with respect to the initial value
(Theorem \ref{thm-reg-v}), the source term (Theorem \ref{thm-reg-w}), and
Lipschitz continuous dependence of the solutions on coefficients and orders
(Theorem \ref{thm-Lip-coef}) so that stability and
uniqueness follow, and we describe the asymptotic behavior of the solution in
Theorem \ref{thm-asymp}.

To this end, we first fix some general settings and notations. Let $L^2(\Om)$
be a usual $L^2$-space with the inner product $(\,\cdot\,,\,\cdot\,)$ and
$H^1_0(\Om)$, $H^2(\Om)$ denote the Sobolev spaces (see, e.g., Adams
\cite{A75}). The elliptic operator $L$ is defined for
$f\in\cD(-L):=H^2(\Om)\cap H^1_0(\Om)$ as
\[
Lf(x)=\sum_{i,j=1}^d\pa_j(a_{ij}(x)\pa_if(x))+c(x)f(x),\quad x\in\Om,
\]
where $a_{ij}=a_{ji}$ ($1\le i,j\le d$) and $c\le0$ in $\ov\Om\,$. Moreover, it
is assumed that $a_{ij}\in C^1(\ov\Om)$, $c\in C(\ov\Om)$ and there exists a
constant $\de>0$ such that
\[
\de\sum_{i=1}^d\xi_i^2\le\sum_{i,j=1}^da_{ij}(x)\xi_i\xi_j,\quad\forall\,x\in\ov\Om\,,\
\forall\,(\xi_1,\ldots,\xi_d)\in\BR^d.
\]
On the other hand, let $\{\la_n,\vp_n\}_{n=1}^\infty$ be the eigensystem of the
elliptic operator $-L$ such that $0<\la_1<\la_2\le\cdots$, $\la_n\to\infty$ as
$n\to\infty$ and $\{\vp_n\}\subset H^2(\Om)\cap H^1_0(\Om)$ forms an
orthonormal basis of $L^2(\Om)$. Then the fractional power $(-L)^\ga$ is
well-defined for $\ga\in\BR$ (see Pazy \cite{P83}) with
\[
\cD((-L)^\ga)=\left\{f\in
L^2(\Om);\thinspace\sum_{n=1}^\infty\left|\la_n^\ga(f,\vp_n)\right|^2<\infty\right\},
\]
and $\cD((-L)^\ga)$ is a Hilbert space with the norm
\[
\|f\|_{\cD((-L)^\ga)}=\left(
\sum_{n=1}^\infty\left|\la_n^\ga(f,\vp_n)\right|^2\right)^{1/2}.
\]
Also we note that $\cD((-L)^\ga)\subset H^{2\ga}(\Om)$ for $\ga>0$ and
especially $\cD((-L)^{1/2})=H^1_0(\Om)$.

Now we are well-prepared to consider the dependency of the solution $u$ to the
initial-boundary value problem \eqref{eq-gov-u}--\eqref{eq-IC-u} upon the
initial value $a$ and the source term $F$. In view of the superposition
principle, it suffices to deal with the cases $F=0$, $a\ne0$ and $a=0$, $F\ne0$
separately.

\begin{thm}\label{thm-reg-v}
Let $F=0$, $0\le\ga\le1$ and $a\in \cD((-L)^\ga)$, where we interpret
$\f1{1-\ga}=\infty$ if $\ga=1$. Concerning the solution $u$ to the
initial-boundary value problem \eqref{eq-gov-u}--$\eqref{eq-IC-u}$, the
followings hold true.

{\rm(a)} There exists a unique solution $u\in C([0,T];L^2(\Om))\cap
C((0,T];H^2(\Om)\cap H^1_0(\Om))$ to \eqref{eq-gov-u}--\eqref{eq-IC-u}.
Actually, $u\in L^{\f1{1-\ga}}(0,T;H^2(\Om)\cap H^1_0(\Om))$ and there exists a
constant $C>0$ such that
\begin{align}
& \|u\|_{C([0,T];L^2(\Om))}\le C\|a\|_{L^2(\Om)},\label{eq-est-v1}\\
& \|u(\,\cdot\,,t)\|_{H^2(\Om)}\le
C\|a\|_{\cD((-L)^\ga)}t^{\al_1(\ga-1)},\quad0<t\le T.\label{eq-est-v2}
\end{align}

{\rm(b)} We have
\begin{equation}\label{eq-asymp-v-init}
\lim_{t\to0}\|u(\,\cdot\,,t)-a\|_{\cD((-L)^\ga)}=0.
\end{equation}

{\rm(c)} There holds $\pa_tu\in C((0,T];L^2(\Om))$. Moreover, there exists a
constant $C>0$ such that
\begin{equation}\label{eq-est-v_t}
\|\pa_tu(\,\cdot\,,t)\|_{L^2(\Om)}\le
C\|a\|_{\cD((-L)^\ga)}t^{\al_1\ga-1},\quad0<t\le T.
\end{equation}

{\rm(d)} If $\ga>0$, then $\pa_t^\be u\in L^{\f1{1-\ga}}(0,T;L^2(\Om))$ for
$0<\be\le\al_1$. Moreover, for $0<\be<1$, there exists a constant $C>0$ such
that
\begin{equation}\label{eq-est-v_t^be}
\|\pa_t^\be u(\,\cdot\,,t)\|_{L^2(\Om)}\le
C\|a\|_{\cD((-L)^\ga)}t^{\al_1\ga-\be},\quad0<t\le T.
\end{equation}
\end{thm}

\begin{thm}\label{thm-reg-w}
Let $a=0$, $1\le p\le\infty$, $0\le\ga\le1$ and $F\in L^p(0,T;\cD((-L)^\ga))$,
where we interpret $1/p=0$ if $p=\infty$. Concerning the solution $u$ to the
initial-boundary value problem \eqref{eq-gov-u}--\eqref{eq-IC-u}, the
followings hold true.

{\rm(a)} If $p=2$, then there exists a unique solution $u\in
L^2(0,T;\cD((-L)^{\ga+1}))$ to \eqref{eq-gov-u}--\eqref{eq-IC-u}. Moreover,
there exists a constant $C>0$ such that
\begin{equation}\label{eq-est-w-L2}
\|u\|_{L^2(0,T;\cD((-L)^{\ga+1}))}\le C\|F\|_{L^2(0,T;\cD((-L)^\ga))}.
\end{equation}

{\rm(b)} If $p\ne2$, then there exists a unique solution $u\in
L^p(0,T;\cD((-L)^{\ga+1-\tau}))$ to \eqref{eq-gov-u}--\eqref{eq-IC-u} for any
$\tau\in(0,1]$. Moreover, there exists a constant $C>0$ such that
\begin{equation}\label{eq-est-w-Lp}
\|u\|_{L^p(0,T;\cD((-L)^{\ga+1-\tau}))}\le\f
C\tau\|F\|_{L^p(0,T;\cD((-L)^\ga))}.
\end{equation}

{\rm(c)} If $\al_1p>1$, then for any $\tau\in(\f1{\al_1p},1]$, there holds
\begin{equation}\label{eq-asymp-w-init}
\lim_{t\to0}\|u(\,\cdot\,,t)\|_{\cD((-L)^{\ga+1-\tau})}=0.
\end{equation}
\end{thm}

\begin{rem}
We compare the conclusions in Theorems \ref{thm-reg-v}--\ref{thm-reg-w} with
those of single-term cases obtained in \cite{SY11}. In case of the homogeneous
source term, i.e. $F=0$ in \eqref{eq-gov-u}, it turns out that Theorem
\ref{thm-reg-v} is a parallel extension of its single-term counterpart. For
instance, in Theorem \ref{thm-reg-v} the regularity results for initial values
$a\in L^2(\Om)$, $a\in H^1_0(\Om)$ and $a\in H^2(\Om)\cap H^1_0(\Om)$ agree
with those in \cite[Theorem 2.1]{SY11}. Especially, it will be readily seen
from the proof of Theorem \ref{thm-reg-v} that the regularity of the solution
$u$ at any positive time can be improved from the initial regularity by $2$
orders in space, namely, $u(\,\cdot\,,t)\in\cD((-L)^{\ga+1})$ if
$a\in\cD((-L)^\ga)$ for $0<t\le T$.

On the other hand, if the source term $F$ does not vanish, the improvement of
regularity in space is strictly less than $2$ orders except for the special
case that $F$ is $L^2$ in time. For example, if $F\in L^2(\Om\times(0,T))$,
then it follows from Theorem \ref{thm-reg-w}(a) that $u\in L^2(0,T;H^2(\Om)\cap
H^1_0(\Om))$, which coincides with \cite[Theorem 2.2]{SY11}. However, if $F\in
L^p(0,T;L^2(\Om))$ with $p\ne2$, then Theorem \ref{thm-reg-w}(b) asserts $u\in
L^p(0,T;\cD((-L)^{1-\tau}))$, where $\tau\in(0,1]$ can be arbitrarily small but
is never zero. The technical reason is that only in case of $p=2$ one can take
advantage of a newly established property in Bazhlekova \cite{B13} (see Lemma
\ref{lem-mono}).
\end{rem}

On basis of these established results, we can consider the dependency of the
solution upon some specified coefficients, especially the orders of Caputo
derivatives. More precisely, we evaluate the difference between the solutions
$u$ and $\wt u$ to
\begin{equation}\label{eq-gov-u1}
\begin{cases}
\!\begin{alignedat}{2}
& \sum_{j=1}^mq_j\pa_t^{\al_j}u=L_Du & \quad & \mbox{in }\Om\times(0,T],\\
& u=0 & \quad & \mbox{on }\pa\Om\times(0,T],\\
& u|_{t=0}=a & \quad & \mbox{in }\Om
\end{alignedat}
\end{cases}
\end{equation}
and
\begin{equation}\label{eq-gov-u2}
\begin{cases}
\!\begin{alignedat}{2}
& \sum_{j=1}^m\wt q_j\pa_t^{\wt\al_j}\wt u=L_{\wt D}\wt u & \quad & \mbox{in
}\Om\times(0,T],\\
& \wt u=0 & \quad & \mbox{on }\pa\Om\times(0,T],\\
& \wt u|_{t=0}=a & \quad & \mbox{in }\Om
\end{alignedat}
\end{cases}
\end{equation}
respectively, where $L_Du(x,t):=\mathrm{div}(D(x)\nb u(x,t))$ and $D$ denotes
the diffusion coefficient. To this end, we fix $1>\ov\al>\underline\al>0$, $\ov
q>\underline q>0$, $\de>0$, $M>0$ and restrict the coefficients in the
admissible sets
\begin{equation}\label{eq-adset}
\begin{aligned}
\cA &
:=\{(\al_1,\ldots,\al_m)\in\BR^m;\thinspace\ov\al\ge\al_1>\al_2>\cdots>\al_m\ge\underline\al\},\\
\cQ & :=\{(q_1,\ldots,q_m)\in\BR^m;\thinspace q_1=1,\ q_j\in[\underline q,\ov
q]\ (j=2,\ldots,m)\},\\
\cU & :=\{D\in C^1(\ov\Om);\thinspace D\ge\de\mbox{ in }\ov\Om\,,\
\|D\|_{C^1(\ov\Om)}\le M\}.
\end{aligned}
\end{equation}
Under these settings, we can show the following result on the Lipschitz
stability.

\begin{thm}\label{thm-Lip-coef}
Fix $\ga,\tau\in(0,1]$. Let $u$ and $\wt u$ be the solutions to
\eqref{eq-gov-u1} and \eqref{eq-gov-u2} respectively, where
\[
a\in\cD((-L)^\ga),\quad(\al_1,\ldots,\al_m),(\wt\al_1,\ldots,\wt\al_m)\in\cA,
\quad(q_1,\ldots,q_m),(\wt
q_1,\ldots,\wt q_m)\in\cQ,\quad D,\wt D\in\cU
\]
and $\cA,\ \cQ,\ \cU$ are defined in \eqref{eq-adset}. Then there exists a
constant $C>0$ depending only on $a,\ \cA,\ \cQ$ and $\cU$ such that
\begin{equation}\label{eq-Lip-coef}
\|u-\wt u\|_{L^{\f1{1-\ga}}(0,T;\cD((-L)^{1-\tau}))}\le\f
C\tau\left(\sum_{j=1}^m|\al_j-\wt\al_j|+\sum_{j=2}^m|q_j-\wt q_j|+\|D-\wt
D\|_{C^1(\ov\Om)}\right)
\end{equation}
for $0<\ga<\f12$ and
\begin{equation}\label{eq-Lip-coef-sp}
\|u-\wt u\|_{L^2(0,T;H^2(\Om))}\le
C\left(\sum_{j=1}^m|\al_j-\wt\al_j|+\sum_{j=2}^m|q_j-\wt q_j|+\|D-\wt
D\|_{C^1(\ov\Om)}\right)
\end{equation}
for $\ga\ge\f12$.
\end{thm}

The above theorem extends a similar result in \cite{LZJY13} for the single-term
case. It is also fundamental for the optimization method for an inverse problem
of determining $\alpha_j,q_j,D(x)$ by extra data of the solution.

In Sakamoto \& Yamamoto \cite{SY11}, the decay rate of the solution to the
single-term time-fractional diffusion equation (1.4)
was shown to be $t^{-\al}$ as
$t\to\infty$. Here we give generalization for the multi-term case where we
specify the principal term of the solution as $t \to \infty$.

\begin{thm}\label{thm-asymp}
Let $F=0$ and $a\in L^2(\Om)$. Then there exists a unique solution $u\in
C([0,\infty);L^2(\Om))\cap C((0,\infty);H^2(\Om)\cap H^1_0(\Om))$ to
\eqref{eq-gov-u}--\eqref{eq-IC-u}. Moreover, there exists a constant $C>0$ such
that
\begin{equation}\label{eq-asymp}
\left\|u(\,\cdot\,,t)-\f{(-L)^{-1}(q_ma)}{\Ga(1-\al_m)\,t^{\al_m}}\right\|
_{H^2(\Om)}\le\f{C\|a\|_{L^2(\Om)}}{t^{\al_{m-1}}}\quad
\mbox{as $t \to \infty$}.
\end{equation}
\end{thm}

\begin{rem}
We explain the significance of the Theorem \ref{thm-asymp}. It reveals that the
decay rate of $u(\,\cdot\,,t)$ in sense of $H^2(\Om)$ is exactly $t^{-\al_m}$
as $t\to\infty$. In fact, inequality \eqref{eq-asymp} implies that there exist
constants $C_2>C_1>0$ such that
\begin{equation}\label{eq-equiv}
C_1\|a\|_{L^2(\Om)}t^{-\al_m}\le\|u(\,\cdot\,,t)\|_{H^2(\Om)}\le
C_2\|a\|_{L^2(\Om)}t^{-\al_m} \quad
\mbox{as $t \to \infty$}.
\end{equation}
Consequently, it turns out that the decay rate $t^{-\al_m}$ is the best
possible. In other words, if
\[
\|u(\,\cdot\,,t)\|_{H^2(\Om)}\le C\,t^{-\be}
\quad \mbox{as $t \to \infty$}
\]
for any order $\be>\al_m$ and some constant $C>0$, then $u(x,t)=0$ for
$x\in\Om$ and $t>0$. Actually, in this case it is easily inferred from the
lower bound in \eqref{eq-equiv} that there should be $a=0$ in $\Om$. Therefore,
Theorem \ref{thm-reg-v} and the upper bound in \eqref{eq-equiv} immediately
imply $u\equiv0$ in $\Om\times(0,\infty)$. Furthermore, \eqref{eq-asymp} also
gives the convergence rate of the approximation
\[u(\,\cdot\,,t)-\f{(-L)^{-1}(q_ma)}{\Ga(1-\al_m)\,t^{-\al_m}}\to0\quad\mbox{in
}H^2(\Om)\mbox{ as }t\to\infty,\]
that is, $t^{-\al_{m-1}}$.
\end{rem}

\Section{Proofs of Main Results}\label{sec-proof}

In this section, we give proofs for the theorems stated in Section
\ref{sec-main}.

In the discussion of single-term time-fractional diffusion equations, it turns
out that the solutions can be explicitly represented by the usual
Mittag-Leffler function
\begin{equation}\label{eq-def-ML}
E_{\al,\be}(z):=\sum_{k=0}^\infty\f{z^k}{\Ga(\al k+\be)},\quad z\in\BC,\
\al>0,\ \be\in\BR,
\end{equation}
and several basic properties play remarkable roles especially for obtaining
estimates for the stability. Since explicit solutions to the multi-term case
are also available by using a generalized form of \eqref{eq-def-ML} called the
multinomial Mittag-Leffler function, we shall first investigate this
generalization so that similar arguments are still feasible for multi-term
time-fractional diffusion equations.

\Subsection{Properties of multinomial Mittag-Leffler
functions}\label{sec-multiML}

The multinomial Mittag-Leffler function is defined as (see Luchko \& Gorenflo
\cite{LG99})
\begin{equation}\label{eq-def-multiML}
E_{(\be_1,\ldots,\be_m),\be_0}(z_1,\ldots,z_m):=\sum_{k=0}^\infty\sum_{k_1+\cdots+k_m=k}\f{(k;k_1,\ldots,k_m)\prod_{j=1}^mz_j^{k_j}}{\Ga(\be_0+\sum_{j=1}^m\be_jk_j)},
\end{equation}
where we assume $0<\be_0<2$, $0<\be_j<1$, $z_j\in\BC$ ($j=1,\ldots,m$), and
$(k;k_1,\ldots,k_m)$ denotes the multinomial coefficient
\[
(k;k_1,\ldots,k_m):=\f{k!}{k_1!\cdots k_m!}\quad\mbox{with }k=\sum_{j=1}^mk_j,
\]
where $k_j$, $1\le j \le m$, are non-negative integers. We recall the following
formula for multinomial coefficients (see Berge \cite{B71})
\begin{equation}\label{eq-multi-coef}
\sum_{j=1}^m(k-1;k_1,\ldots,k_{j-1},k_j-1,k_{j+1},\ldots,k_m)=(k;k_1,\ldots,k_m).
\end{equation}
If some $k_{j_0}$ vanishes, we understand
$(k-1;k_1,\ldots,k_{j_0-1},k_{j_0}-1,k_{j_0+1},\ldots,k_m)=0$ and
\eqref{eq-multi-coef} degenerates to its lower dimensional version.

Concerning the relation between multinomial Mittag-Leffler functions with
different parameters, we have the following lemma.

\begin{lem}\label{lem-multiML-alter}
Let $0<\be_0<2,\ 0<\be_j<1\ (j=1,\ldots,m)$ and $z_j\in\BC\ (j=1,\ldots,m)$ be
fixed. Then
\[
\f1{\Ga(\be_0)}+\sum_{j=1}^mz_jE_{(\be_1,\ldots,\be_m),\be_0+\be_j}(z_1,\ldots,z_m)=E_{(\be_1,\ldots,\be_m),\be_0}(z_1,\ldots,z_m).
\]
\end{lem}

\begin{proof}
According to definition \eqref{eq-def-multiML}, direct calculations yield
\begin{align}
&
\quad\,\,\sum_{j=1}^mz_jE_{(\be_1,\ldots,\be_m),\be_0+\be_j}(z_1,\ldots,z_m)\nonumber\\
& =\sum_{j=1}^m\sum_{k=0}^\infty\sum_{k_1+\cdots+k_m=k}\f{(k;k_1,\ldots,k_m)\,z_j\prod_{\ell=1}^mz_\ell^{k_\ell}}{\Ga(\be_0+\be_j+\sum_{\ell=1}^m\be_\ell k_\ell)}\nonumber\\
& =\sum_{k=0}^\infty\sum_{j=1}^m\left\{\f{z_j^{k+1}}{\Ga(\be_0+\be_j(k+1))}+\sum_{\substack{k_1+\cdots+k_m=k\\k_j<k}}\f{(k;k_1,\ldots,k_m)\,z_j\prod_{\ell=1}^mz_\ell^{k_\ell}}{\Ga(\be_0+\be_j+\sum_{\ell=1}^m\be_\ell k_\ell)}\right\}\label{eq-pf-multiML-1}\\
& =\sum_{k=0}^\infty\sum_{j=1}^m\left\{\f{z_j^{k+1}}{\Ga(\be_0+\be_j(k+1))}+\sum_{\substack{k_1+\cdots+k_m=k+1\\0<k_j<k+1}}\f{(k;k_1,\ldots,k_{j-1},k_j-1,k_{j+1},\ldots,k_m)\,\prod_{\ell=1}^mz_\ell^{k_\ell}}{\Ga(\be_0+\sum_{\ell=1}^m\be_\ell k_\ell)}\right\}\nonumber\\
&
=\sum_{k=0}^\infty\left\{\sum_{j=1}^m\f{z_j^{k+1}}{\Ga(\be_0+\be_j(k+1))}+\sum_{\substack{k_1+\cdots+k_m=k+1\\k_\ell<k+1\ (\forall\,\ell)}}\f{(k+1;k_1,\ldots,k_m)\,\prod_{\ell=1}^mz_\ell^{k_\ell}}{\Ga(\be_0+\sum_{\ell=1}^m\be_\ell k_\ell)}\right\}\label{eq-pf-multiML-2}\\
&
=\sum_{k=0}^\infty\sum_{k_1+\cdots+k_m=k+1}\f{(k+1;k_1,\ldots,k_m)\prod_{\ell=1}^mz_j^{k_j}}{\Ga(\be_0+\sum_{\ell=1}^m\be_\ell k_\ell)}\nonumber\\
&
=\sum_{k=1}^\infty\sum_{k_1+\cdots+k_m=k}\f{(k;k_1,\ldots,k_m)\prod_{\ell=1}^mz_\ell^{k_\ell}}{\Ga(\be_0+\sum_{\ell=1}^m\be_\ell k_\ell)}=E_{(\be_1,\ldots,\be_m),\be_0}(z_1,\ldots,z_m)-\f1{\Ga(\be_0)},\nonumber
\end{align}
where we apply formula \eqref{eq-multi-coef} to obtain \eqref{eq-pf-multiML-2}.
In \eqref{eq-pf-multiML-1}, we distill the case $k_j=k$ in the $j$-th term and
substitute $k_j+1$ with $k_j$ for the others to proceed to the next equality.
\end{proof}

Concerning the regularity of the solution to a single-term time-fractional
diffusion equation, the estimate (see \cite[p. 35]{P99})
\[
|E_{\al,\be}(-\eta)|\le \f C{1+\eta},\quad\eta\ge0
\]
is essential. Here we extend the above inequality to the multinomial case by a
complex variable argument.

\begin{lem}\label{lem-multiML-est}
Let $0<\be<2$ and $1>\al_1>\cdots>\al_m>0$ be given. Assume that
$\al_1\pi/2<\mu<\al_1\pi,\ \mu\le|\mathrm{arg}(z_1)|\le\pi$ and there exists
$K>0$ such that $-K\le z_j<0\ (j=2,\ldots,m)$. Then there exists a constant
$C>0$ depending only on $\mu,\ K,\ \al_j\ (j=1,\ldots,m)$ and $\be$ such that
\[
|E_{(\al_1,\al_1-\al_2,\ldots,\al_1-\al_m),\be}(z_1,\ldots,z_m)|\le\f C{1+|z_1|}.
\]
\end{lem}

\begin{proof}
Let $\al_j,z_j$ ($j=1,\ldots,m$) and $\be$ be assumed as above and introduce
the notation
\[
E_{\bm\al',\be}(z_1,\ldots,z_m):=E_{(\al_1,\al_1-\al_2,\ldots,\al_1-\al_m),\be}(z_1,\ldots,z_m).
\]
In the sequel, we denote by $C$ a general positive constant depending at most
on $\mu$, $K$, $\al_j$ ($j=1,\ldots,m$) and $\be$. First we rewrite the
multinomial Mittag-Leffler function \eqref{eq-def-multiML} in an alternative
form with the aid of the contour integral representation of $1/\Ga(z)$ (see
\cite[\S1.1.6]{P99}) that
\[
\f1{\Ga(z)}=\f1{2\al_1\pi\,\ri}\int_{\ga(R,\te)}\exp(\ze^{1/\al_1})\ze^{(1-z-\al_1)/\al_1}\,\rd\ze,
\]
where $R>0$ is a constant to be determined later and $\al_1\pi/2<\te<\mu$. Here
$\ga(R,\te)$ denotes the contour
\[
\ga(R,\te):=\{\ze\in\BC;\thinspace|\ze|=R,\
|\arg(\ze)|\le\te\}\cup\{\ze\in\BC;\thinspace|\ze|>R,\ |\arg(\ze)|=\pm\te\}.
\]
Then it follows from the multinomial formula that
\begin{align*}
& \quad\;E_{\bm\al',\be}(z_1,\ldots,z_m)\\
&
=\f1{2\al_1\pi\,\ri}\sum_{k=0}^\infty\sum_{k_1+\cdots+k_m=k}(k;k_1,\ldots,k_m)\prod_{j=1}^mz_j^{k_j}\\
& \qquad\qquad\qquad\qquad\qquad\;\;\;\times\left\{\int_{\ga(R,\te)}\exp(\ze^{1/\al_1})\ze^{(1-\be-\al_1(k+1)-\al_2k_2-\cdots-\al_mk_m)/\al_1}\,\rd\ze\right\}\\
&
=\f1{2\al_1\pi\,\ri}\int_{\ga(R,\te)}\exp(\ze^{1/\al_1})\ze^{(1-\be)/\al_1-1}\\
& \qquad\qquad\qquad\quad\;\times\sum_{k=0}^\infty\left\{\sum_{k_1+\cdots+k_m=k}(k;k_1,\ldots,k_m)\left(\f{z_1}\ze\right)^{k_1}\prod_{j=2}^m\left(\f{z_j}{\ze^{1-\al_j/\al_1}}\right)^{k_j}\right\}\rd\ze\\
& =\f1{2\al_1\pi\,\ri}\int_{\ga(R,\te)}\exp(\ze^{1/\al_1})\ze^{(1-\be)/\al_1-1}\sum_{k=0}^\infty\left(\f{z_1}\ze+\sum_{j=2}^m\f{z_k}{\ze^{1-\al_j/\al_1}}\right)^k\,\rd\ze.
\end{align*}
In order to guarantee the convergence of the summation with respect to $k$, it
is required that
\[
\left|\f{z_1}\ze+\sum_{j=2}^m\f{z_j}{\ze^{1-\al_j/\al_1}}\right|<1,\quad\forall\,\ze\in\ga(R,\te).
\]
Since $|z_j|\le K$ for $j=2,\ldots,m$, the above inequality is achieved by
taking $R$ such that
\[
R>|z_1|+K\sum_{j=2}^mR^{\al_j/\al_1}.
\]
Moreover, if we restrict, for example, $|z_1|\le K$, then $R$ can be fixed as a
constant depending only on $K$ and $\al_j$ ($j=1,\ldots,m$). Now we deduce for
$|z_j|\le K$ ($j=1,\ldots,m$) that
\begin{equation}\label{eq-MML-contour}
E_{\bm\al',\be}(z_1,\ldots,z_m)=\f1{2\al_1\pi\,\ri}\int_{\ga(R,\te)}\f{\exp(\ze^{1/\al_1})\ze^{(1-\be)/\al_1}}{\ze-z_1-\sum_{j=2}^mz_j\ze^{\al_j/\al_1}}\,\rd\ze.
\end{equation}

Next we fix $z_2,\ldots,z_m$ as negative parameters and regard both sides of
\eqref{eq-MML-contour} as functions of the single complex variable $z_1$, which
allows the application of the principle of analytic continuation to extend
equality \eqref{eq-MML-contour} to a domain including
$\{z_1\in\BC;\thinspace\mu\le|\arg(z_1)|\le\pi\}$ (see Figure
\ref{fig-multiML-est}).
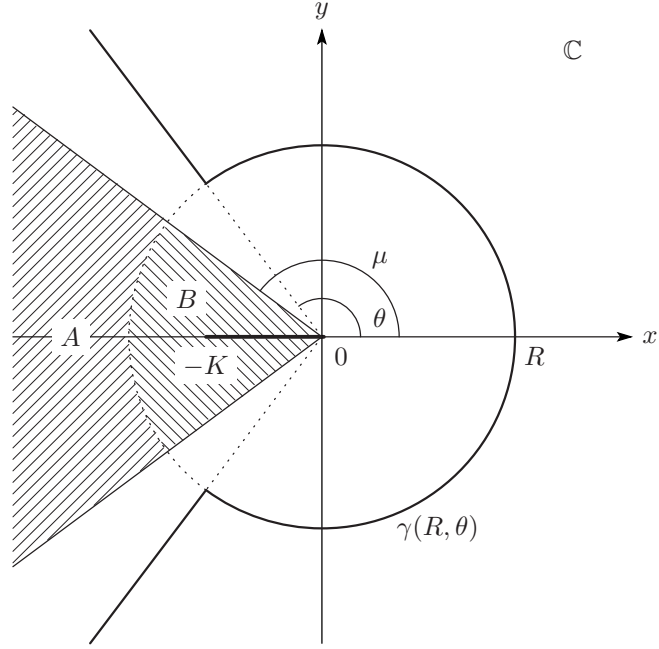
\begin{figure}[htbp]\centering
\input{fig1.tex}\\
\caption{Settings of Lemma \ref{lem-multiML-est} and the contour $\ga(R,\te)$.
If $z_1$ is located in the shaded domain $A$, we employ the principle of
analytic continuation and the contour integral representation
\eqref{eq-MML-contour}. When $z_1$ is in the shaded domain $B$, it suffices to
argue by definition \eqref{eq-def-multiML}.}\label{fig-multiML-est}
\end{figure}

For $|z_1|>R$, we investigate the denominator of the integrand in
\eqref{eq-MML-contour}. Since $z_j<0$ and $\al_j<\al_1$ for $j=2,\ldots,m$, it
turns out that the curve $\ze-\sum_{j=2}^mz_j\ze^{\al_j/\al_1}$
($\ze\in\ga(R,\te)$) locates on the right-hand side of $\ga(R,\te)$; that is,
$\ga(R,\te)$ is shifted by the term $-\sum_{j=2}^mz_j\ze^{\al_j/\al_1}$ to the
positive direction. This observation immediately implies
\[
\min_{\ze\in\ga(R,\te)}\left|\ze-z_1-\sum_{j=2}^mz_j\ze^{\al_j/\al_1}\right|\ge\min_{\ze\in\ga(R,\te)}|\ze-z_1|\ge|z_1|\sin(\mu-\te).
\]
Therefore, we come up with the estimate
\begin{align*}
|E_{\bm\al',\be}(z_1,\ldots,z_m)| & =\f1{2\al_1\pi}\left|\int_{\ga(R,\te)}\f{\exp(\ze^{1/\al_1})\ze^{(1-\be)/\al_1}}{\ze-z_1-\sum_{j=2}^mz_j\ze^{\al_j/\al_1}}\,\rd\ze\right|\\
&
\le\left(\f1{2\al_1\pi\sin(\mu-\te)}\int_{\ga(R,\te)}|\exp(\ze^{1/\al_1})||\ze^{(1-\be)/\al_1}|\,\rd\ze\right)\f1{|z_1|}.
\end{align*}
The integral along $\ga(R,\te)$ converges, because for $\ze$ such that
$\arg(\ze)=\pm\te$ and $|\ze|>R$, there holds
\[
|\exp(\ze^{1/\al_1})|=\exp(|\ze|^{1/\al_1}\cos(\te/\al_1))\quad\mbox{with
}\cos(\te/\al_1)<0,
\]
while the integral on the arc $\{\ze\in\BC;\thinspace|\ze|=R,\
|\arg(\ze)|\le\te\}$ is a constant. Consequently
\begin{equation}\label{eq-MML-est-L}
|E_{\bm\al',\be}(z_1,\ldots,z_m)|\le\f C{|z_1|},\quad\mu\le|\arg(z_1)|\le\pi,\
|z_1|>R.
\end{equation}

For $\mu\le|\arg(z_1)|\le\pi$ such that $|z_1|\le R$, it is directly verified
that
\begin{align*}
|E_{\bm\al',\be}(z_1,\ldots,z_m)| & =\left|\sum_{k=0}^\infty\sum_{k_1+\cdots+k_m=k}\f{(k;k_1,\ldots,k_m)\prod_{j=1}^mz_j^{k_j}}{\Ga(\be+\al_1k-\sum_{j=2}^m\al_jk_j)}\right|\\
&
\le\sum_{k=0}^\infty\sum_{k_1+\cdots+k_m=k}\f{(k;k_1,\ldots,k_m)\prod_{j=1}^m|z_j|^{k_j}}{\Ga(\be+\al_1k-\sum_{j=2}^m\al_jk_j)}\\
& \le
C\sum_{k=0}^\infty\sum_{k_1+\cdots+k_m=k}\f{(k;k_1,\ldots,k_m)\prod_{j=1}^m|z_j|^{k_j}}{\Ga(\be+(\al_1-\al_2)k)}\\
&
=C\sum_{k=0}^\infty\f1{\Ga(\be+(\al_1-\al_2)k)}\left(\sum_{j=1}^m|z_j|\right)^k\le
C\sum_{k=0}^\infty\f{(R+(m-1)K)^k}{\Ga(\be+(\al_1-\al_2)k)}\le C,
\end{align*}
which, together with \eqref{eq-MML-est-L}, finishes the proof.
\end{proof}

For later use, we adopt the abbreviation
\begin{equation}\label{eq-abb-multiML}
E_{\bm\al',\be}^{(n)}(t):=E_{(\al_1,\al_1-\al_2,\ldots,\al_1-\al_m),\be}(-\la_nt^{\al_1},-q_2t^{\al_1-\al_2},\ldots,-q_mt^{\al_1-\al_m}),\quad
t>0,
\end{equation}
where $\la_n$ is the $n$-th eigenvalue of $-L$, $0<\be<2$, and $\al_j$, $q_j$
are those positive constants in \eqref{eq-gov-u}. Especially, regarding the
derivative of $t^{\al_1}E_{\bm\al',1+\al_1}^{(n)}(t)$ with respect to $t>0$, we
state the following technical lemma.

\begin{lem}\label{lem-multiML-diff}
Let $1>\al_1>\cdots>\al_m>1$. Then
\[
\f\rd{\rd
t}\left\{t^{\al_1}E_{\bm\al',1+\al_1}^{(n)}(t)\right\}=t^{\al_1-1}E_{\bm\al',\al_1}^{(n)}(t),\quad
t>0.
\]
\end{lem}

\begin{proof}
By definition, we carry out a direct differentiation and utilize the formula
$\Ga(s)=\Ga(s+1)/s$ to derive
\begin{align*}
& \quad\;\f\rd{\rd t}\left\{t^{\al_1}E_{\bm\al',1+\al_1}^{(n)}(t)\right\}\\
& =\f\rd{\rd
t}\left\{\sum_{k=0}^\infty\sum_{k_1+\cdots+k_m=k}\f{(k;k_1,\ldots,k_m)(-\la_n)^{k_1}\prod_{j=2}^m(-q_j)^{k_j}t^{\al_1(k+1)-\al_2k_2-\cdots-\al_mk_m}}{\Ga(1+\al_1(k+1)-\sum_{j=2}^m\al_jk_j)}\right\}\\
&
=\sum_{k=0}^\infty\sum_{k_1+\cdots+k_m=k}\f{(k;k_1,\ldots,k_m)(-\la_n)^{k_1}\prod_{j=2}^m(-q_j)^{k_j}t^{\al_1(k+1)-\al_2k_2-\cdots-\al_mk_m-1}}{\Ga(\al_1(k+1)-\sum_{j=2}^m\al_jk_j)}\\
&
=t^{\al_1-1}\sum_{k=0}^\infty\sum_{k_1+\cdots+k_m=k}\f{(k;k_1,\ldots,k_m)(-\la_nt^{\al_1})^{k_1}\prod_{j=2}^m(-q_jt^{\al_1-\al_j})^{k_j}}{\Ga(\al_1+\al_1k_1+\sum_{j=2}^m(\al_1-\al_j)k_j)}\\
& =t^{\al_1-1}E_{\bm\al',\al_1}^{(n)}(t).
\end{align*}
Here we use the fact that $t^{\al_1}E_{\bm\al',1+\al_1}^{(n)}(t)$ is real
analytic for $t>0$ so that termwise differentiations are available.
\end{proof}

\Subsection{Proofs of Theorems
\ref{thm-reg-v}--\ref{thm-Lip-coef}}\label{sec-stability}

Now we are ready to employ the multinomial Mittag-Leffler functions to show
results on the well-posedness. For later use we recall the eigensystem
$\{\la_n,\vp_n\}$ of the elliptic operator $-L$ and the abbreviation
$E_{\bm\al',\be}^{(n)}(t)$ ($0<\be<2$) in \eqref{eq-abb-multiML}.

First we prove Theorem \ref{thm-reg-v}, that is, the case of vanishing source
term $F$. It was shown in \cite{L11} that the explicit solution to
\eqref{eq-gov-u}--\eqref{eq-IC-u} is given by
\begin{equation}\label{eq-rep-v}
u(\,\cdot\,,t)=\sum_{n=1}^\infty\left(1-\la_nt^{\al_1}E_{\bm\al',1+\al_1}^{(n)}(t)\right)(a,\vp_n)\,\vp_n.
\end{equation}
With the aid of Lemmata \ref{lem-multiML-alter}--\ref{lem-multiML-diff}, it is
straightforward to demonstrate the well-posedness by dominating the solution by
the initial value.

\begin{proof}[Proof of Theorem \ref{thm-reg-v}]
Let $a\in\cD((-L)^\ga)$ with $0\le\ga\le1$. In the sequel, by $C$ we refer to
positive constants independent of the initial value $a$ which may vary from
line by line.

(a) First, a direct application of Lemma \ref{lem-multiML-est} yields
\[
\left|1-\la_nt^{\al_1}E_{\bm\al',1+\al_1}^{(n)}(t)\right|\le1+\la_nt^{\al_1}\f
C{1+\la_nt^{\al_1}}\le C.
\]
Thus, we take advantage of \eqref{eq-rep-v} to derive
\begin{equation}\label{eq-est-v-L2}
\|u(\,\cdot\,,t)\|_{L^2(\Om)}=\left\{\sum_{n=1}^\infty\left|1-\la_nt^{\al_1}E_{\bm\al',1+\al_1}^{(n)}(t)\right|^2|(a,\vp_n)|^2\right\}^{1/2}\le
C\|a\|_{L^2(\Om)}
\end{equation}
for $0<t\le T$, where we use the fact that $\{\vp_n\}$ forms an orthonormal
basis of $L^2(\Om)$. Since the summation in \eqref{eq-rep-v} converges in
$L^2(\Om)$ uniformly in $t\in[0,T]$, we get $u\in C([0,T];L^2(\Om))$ or
\eqref{eq-est-v1}. Furthermore, by the definition of $\cD(-L)$, we see
\[
\|u(\,\cdot\,,t)\|_{\cD(-L)}^2=\sum_{n=1}^\infty\left(\la_n\left|1-\la_nt^{\al_1}E_{\bm\al',1+\al_1}^{(n)}(t)\right|\right)^2|(a,\vp_n)|^2.
\]
In order to treat the term $1-\la_nt^{\al_1}E_{\bm\al',1+\al_1}^{(n)}(t)$, we
substitute
\[
\be_0=1,\quad\be_1=\al_1,\quad z_1=-\la_nt^{\al_1},\quad\be_j=\al_1-\al_j\mbox{
and }z_j=-q_jt^{\al_1-\al_j}\ (j=2,\ldots,m)
\]
in Lemma \ref{lem-multiML-alter} and then utilize Lemma \ref{lem-multiML-est}
to deduce
\begin{align*}
\left|1-\la_nt^{\al_1}E_{\bm\al',1+\al_1}^{(n)}(t)\right| &
=\left|E_{\bm\al',1}^{(n)}(t)+\sum_{j=2}^mq_jt^{\al_1-\al_j}E_{\bm\al',1+\al_1-\al_j}^{(n)}(t)\right|\\
&
\le\left|E_{\bm\al',1}^{(n)}(t)\right|+C\sum_{j=2}^mt^{\al_1-\al_j}\left|E_{\bm\al',1+\al_1-\al_j}^{(n)}(t)\right|\le C\sum_{j=1}^m\f{t^{\al_1-\al_j}}{1+\la_nt^{\al_1}}.
\end{align*}
Therefore, for $0<t\le T$, we estimate
\begin{align*}
\|u(\,\cdot\,,t)\|_{\cD(-L)}^2 &
=\sum_{n=1}^\infty\left|\la_n^{1-\ga}\left(1-\la_nt^{\al_1}E_{\bm\al',1+\al_1}^{(n)}(t)\right)\right|^2|\la_n^\ga(a,\vp_n)|^2\\
& \le
C^2\sum_{n=1}^\infty\left(\sum_{j=1}^m\f{\la_n^{1-\ga}t^{\al_1-\al_j}}{1+\la_nt^{\al_1}}\right)^2|\la_n^\ga(a,\vp_n)|^2\\
& \le
C^2\sum_{n=1}^\infty\left(\sum_{j=1}^m\f{(\la_nt^{\al_1})^{1-\ga}}{1+\la_nt^{\al_1}}t^{\al_1\ga-\al_j}\right)^2|\la_n^\ga(a,\vp_n)|^2\\
& \le
C^2\left(\sum_{j=1}^mt^{\al_1\ga-\al_j}\right)^2\sum_{n=1}^\infty|\la_n^\ga(a,\vp_n)|^2\le\left(C\|a\|_{\cD((-L)^\ga)}t^{\al_1(\ga-1)}\right)^2,
\end{align*}
where we use the fact
\[\f{(\la_nt^{\al_1})^{1-\ga}}{1+\la_nt^{\al_1}}\le\left\{\!\begin{alignedat}{2}
& \f1{1+\la_nt^{\al_1}} & \quad & \mbox{if }\la_nt^{\al_1}\le1\\
& \f{\la_nt^{\al_1}}{1+\la_nt^{\al_1}} & \quad & \mbox{if }\la_nt^{\al_1}\ge1
\end{alignedat}\right\}\le1\]
in the third inequality. This, together with the fact $\cD(-L)\subset
H^2(\Om)$, yield the estimate \eqref{eq-est-v2}. Furthermore, it follows
immediately from \eqref{eq-est-v2} and $\al_1<1$ that $u\in
L^{\f1{1-\ga}}(0,T;H^2(\Om)\cap H^1_0(\Om))$.

(b) In order to investigate the asymptotic behavior near $t=0$, first we have
\[
\|u(\,\cdot\,,t)-a\|_{\cD((-L)^\ga)}^2=\sum_{n=1}^\infty\left|\la_nt^{\al_1}E_{\bm\al',1+\al_1}^{(n)}(t)\right|^2|\la_n^\ga(a,\vp_n)|^2\le\left(C\|a\|_{\cD((-L)^\ga)}\right)^2<\infty
\]
for $0\le t\le T$ by a direct calculation and Lemma \ref{lem-multiML-est}. On
the other hand, in view of Lemma \ref{lem-multiML-alter}, the term
$\la_nt^{\al_1}E_{\bm\al',1+\al_1}^{(n)}(t)$ can be rewritten as
\[
\la_nt^{\al_1}E_{\bm\al',1+\al_1}^{(n)}(t)=-(E_{\bm\al',1}^{(n)}(t)-1)-\sum_{j=2}^mq_jt^{\al_1-\al_j}E_{\bm\al',1+\al_1-\al_j}^{(n)}(t).
\]
Thanks to the fact that $\lim_{t\to0}(E_{\bm\al',1}^{(n)}(t)-1)=0$ and the
boundedness of $E_{\bm\al',1+\al_1-\al_j}^{(n)}$ ($j=2,\ldots,m$) by Lemma
\ref{lem-multiML-est} for each $n=1,2,\ldots$, the above observation implies
\[
\lim_{t\to0}(\la_nt^{\al_1}E_{\bm\al',1+\al_1}^{(n)}(t))=0,\quad\forall\,n=1,2,\ldots.
\]
Therefore, \eqref{eq-asymp-v-init} follows immediately from Lebesgue's
dominated convergence theorem.

(c) In order to deal with $\pa_tu$, we make use of Lemma \ref{lem-multiML-diff}
to obtain
\[
\pa_tu(\,\cdot\,,t)=-t^{\al_1-1}\sum_{n=1}^\infty\la_nE_{\bm\al',\al_1}^{(n)}(t)(a,\vp_n)\vp_n.
\]
Then a similar argument to that for \eqref{eq-est-v2} indicates
\begin{align*}
\|\pa_tu(\,\cdot\,,t)\|_{L^2(\Om)}^2 &
=t^{2(\al_1-1)}\sum_{n=1}^\infty\left|\la_n^{1-\ga}E_{\bm\al',\al_1}^{(n)}(t)\right|^2|\la_n^\ga(a,\vp_n)|^2\\
& \le
C^2t^{2(\al_1-1)}\sum_{n=1}^\infty\left(\f{(\la_nt^{\al_1})^{1-\ga}}{1+\la_nt^{\al_1}}t^{\al_1(\ga-1)}\right)^2|\la_n^\ga(a,\vp_n)|^2\\
& \le\left(C\|a\|_{\cD((-L)^\ga)}t^{\al_1\ga-1}\right)^2,\quad0<t\le T
\end{align*}
or \eqref{eq-est-v_t}. This implies $\pa_tu\in C((0,T];L^2(\Om))$ immediately.

(d) Finally, to give estimates for $\pa_t^\be u$ with $0<\be<1$ when $\ga>0$,
we employ \eqref{eq-est-v_t} and turn to the definition of the Caputo
derivative to obtain
\begin{align*}
\|\pa_t^\be u(\,\cdot\,,t)\|_{L^2(\Om)} &
=\f1{\Ga(1-\be)}\left\|\int_0^t\f{\pa_su(\,\cdot\,,s)}{(t-s)^\be}\,\rd s\right\|_{L^2(\Om)}\le
C\int_0^t\f{\|\pa_su(\,\cdot\,,s)\|_{L^2(\Om)}}{(t-s)^\be}\,\rd s\\
& \le C\|a\|_{\cD((-L)^\ga)}\int_0^ts^{\al_1\ga-1}(t-s)^{-\be}\,\rd s\le
C\|a\|_{\cD((-L)^\ga)}t^{\al_1\ga-\be},\quad0<t\le T
\end{align*}
or \eqref{eq-est-v_t^be}, where the first inequality follows from Minkowski's
inequality for integrals. Especially, as long as $\be\le\al_1$, there holds
$\al_1\ga-\be>\ga-1$ and obviously $\pa_t^\be v\in
L^{\f1{1-\ga}}(0,T;L^2(\Om))$.

Collecting all the results above, we complete the proof of Theorem
\ref{thm-reg-v}.
\end{proof}

Next we turn to the proof of Theorem \ref{thm-reg-w}, that is, the case of
vanishing initial value $a$. To construct an explicit solution, we apply the
eigenfunction expansion method. In other words, we seek for a solution to
\eqref{eq-gov-u}--\eqref{eq-IC-u} of the particular form
\begin{equation}\label{eq-const}
u(\,\cdot\,,t)=\sum_{n=1}^\infty T_n(t)\vp_n,\quad0<t\le T,
\end{equation}
where $\vp_n$ is the $n$-th eigenfunction of $-L$. The substitution of
\eqref{eq-const} into \eqref{eq-gov-u} yields
\[
\sum_{n=1}^\infty\left(\sum_{j=1}^mq_j\pa_t^{\al_j}T_n(t)\right)\vp_n=-\sum_{n=1}^\infty\la_nT_n(t)\vp_n+\sum_{n=1}^\infty(F(\,\cdot\,,t),\vp_n)\vp_n.
\]
Therefore, it is readily seen from the orthogonality of $\{\vp_n\}$ and the
homogeneous initial condition \eqref{eq-IC-u} that $T_n$ satisfies an initial
value problem for an ordinary differential equation
\[
\begin{cases}
\!\begin{alignedat}{2}
& \sum_{j=1}^mq_j\pa_t^{\al_j}T_n(t)+\la_nT_n(t)=(F(\,\cdot\,,t),\vp_n), &
\quad & 0<t\le T,\\
& T_n(0)=0.
\end{alignedat}
\end{cases}
\]
Then it follows from \cite[Theorem 4.1]{LG99} that
\[
T_n(t)=\int_0^ts^{\al_1-1}E_{\bm\al',\al_1}^{(n)}(s)(F(\,\cdot\,,t-s),\vp_n)\,\rd
s,
\]
implying that the solution takes the form of a convolution
\begin{equation}\label{eq-rep-w}
u(\,\cdot\,,t)=\int_0^tU(s)F(\,\cdot\,,t-s)\,\rd s,
\end{equation}
where
\begin{equation}\label{eq-def-W}
U(t)f:=t^{\al_1-1}\sum_{n=1}^\infty E_{\bm\al',\al_1}^{(n)}(t)(f,\vp_n)\vp_n.
\end{equation}
Before proceeding to the proof, we introduce a key lemma for showing Theorem
\ref{thm-reg-w}(a).

\begin{lem}\label{lem-mono}
{\rm(see \cite[Theorem 3.2]{B13})} The function
$t^{\al_1-1}E_{\bm\al',\al_1}^{(n)}(t)$ is positive for $t>0$.
\end{lem}

\begin{proof}[Proof of Theorem \ref{thm-reg-w}]
Let $F\in L^p(0,T;\cD((-L)^\ga))$ with $1\le p\le\infty$ and $0\le\ga\le1$. In
the sequel, by $C$ we refer to a general positive constant independent of $F$
and $\tau$.

(a) Let $p=2$. According to the expression \eqref{eq-rep-w}--\eqref{eq-def-W},
formally we write
\[\|u(\,\cdot\,,t)\|_{\cD(-L)}^2=\sum_{n=1}^\infty\la_n^2\left(\int_0^ts^{\al_1-1}E_{\bm\al',\al_1}^{(n)}(s)(F(\,\cdot\,,t-s),\vp_n)\,\rd
s\right)^2.\]
Using Young's inequality for convolutions, we estimate
\begin{align*}
\|u\|_{L^2(0,T;\cD(-L))}^2 &
=\sum_{n=1}^\infty\la_n^2\left\|\int_0^ts^{\al_1-1}E_{\bm\al',\al_1}^{(n)}(s)(F(\,\cdot\,,t-s),\vp_n)\,\rd
s\right\|_{L^2(0,T)}^2\\
&
\le\sum_{n=1}^\infty\left(\la_n\int_0^Tt^{\al_1-1}|E_{\bm\al',\al_1}^{(n)}(t)|\,\rd
t\right)^2\|(F(\,\cdot\,,t),\vp_n)\|_{L^2(0,T)}^2.
\end{align*}
By Lemma \ref{lem-mono}, we can remove the absolute value of $E_{\bm\al',\al_1}^{(n)}(t)$ and apply Lemma \ref{lem-multiML-diff} to derive
\[
\int_0^Tt^{\al_1-1}|E_{\bm\al',\al_1}^{(n)}(t)|\,\rd
t=\int_0^Tt^{\al_1-1}E_{\bm\al',\al_1}^{(n)}(t)\,\rd
t=T^{\al_1}E_{\bm\al',1+\al_1}^{(n)}(T).
\]
Consequently, we use Lemma \ref{lem-multiML-est} to conclude
\begin{align*}
\|u\|_{L^2(0,T;H^2(\Om))}^2 & \le C^2\|u\|_{L^2(0,T;\cD(-L))}^2\le
C^2\sum_{n=1}^\infty\left(\la_nT^{\al_1}E_{\bm\al',1+\al_1}^{(n)}(T)\right)^2\|(F(\,\cdot\,,t),\vp_n)\|_{L^2(0,T)}^2\\
& \le
C^2\sum_{n=1}^\infty\|(F(\,\cdot\,,t),\vp_n)\|_{L^2(0,T)}^2=\left(C\|F\|_{L^2(\Om\times(0,T))}\right)^2.
\end{align*}

(b) Fix $\tau\in(0,1]$ arbitrarily. First we give an estimate for
\eqref{eq-def-W} with $f\in\cD((-L)^\ga)$. Similarly to the proof of Theorem
\ref{thm-reg-v}, we apply Lemma \ref{lem-multiML-est} to deduce
\begin{align*}
\|U(t)f\|_{\cD((-L)^{\ga+1-\tau})}^2 &
=t^{2(\al_1-1)}\sum_{n=1}^\infty\left|\la_n^{1-\tau}E_{\bm\al',\al_1}^{(n)}\right|^2|\la_n^\ga(f,\vp_n)|^2\\
& \le
C^2t^{2(\al_1-1)}\sum_{n=1}^\infty\left(\f{(\la_nt^{\al_1})^{1-\tau}}{1+\la_nt^{\al_1}}t^{\al_1(\tau-1)}\right)^2|\la_n^\ga(f,\vp_n)|^2\\
& \le\left(C\|f\|_{\cD((-L)^\ga)}t^{\al_1\tau-1}\right)^2,\quad0<t\le T.
\end{align*}
Using \eqref{eq-rep-w} and Minkowski's inequality for integrals, formally we
have
\begin{align}
\|u(\,\cdot\,,t)\|_{\cD((-L)^{\ga+1-\tau})} &
=\left\|\int_0^tU(s)F(\,\cdot\,,t-s)\,\rd s\right\|_{\cD((-L)^{\ga+1-\tau})}\nonumber\\
& \le\int_0^t\|U(s)F(\,\cdot\,,t-s)\|_{\cD((-L)^{\ga+1-\tau})}\,\rd
s\nonumber\\
& \le C\int_0^t\|F(\,\cdot\,,t-s)\|_{\cD((-L)^\ga)}s^{\al_1\tau-1}\,\rd
s,\quad0<t\le T.\label{eq-est-U}
\end{align}
Finally, it follows from Young's inequality for convolutions that
\begin{align*}
\|u\|_{L^p(0,T;\cD((-L)^{\ga+1-\tau}))} & \le
C\left\|\int_0^t\|F(\,\cdot\,,t-s)\|_{\cD((-L)^\ga)}s^{\al_1\tau-1}\,\rd s\right\|_{L^p(0,T)}\\
& \le C\|F\|_{L^p(0,T;\cD((-L)^\ga))}\int_0^Tt^{\al_1\tau-1}\,\rd t\le\f
C\tau\|F\|_{L^p(0,T;\cD((-L)^\ga))}.
\end{align*}
This completes the verification of \eqref{eq-est-w-Lp}.

(c) Assume $\al_1p>1$ and fix $\tau\in(\f1{\al_1p},1]$ arbitrarily. To
investigate the asymptotic behavior near $t=0$, we apply H\"older's inequality
to \eqref{eq-est-U} to see
\[
\|u(\,\cdot\,,t)\|_{\cD((-L)^{\ga+1-\tau})}\le
C\|F\|_{L^p(0,t;\cD((-L)^\ga))}\left(\int_0^ts^{(\al_1\tau-1)p'}\,\rd s\right)^{1/p'},
\]
where $p'$ is the conjugate number of $p$, i.e. $1/p+1/p'=1$. Since
$\tau>\f1{\al_1p}$, we see $(\al_1\tau-1)p'>-1$ and then
$\lim_{t\to0}\int_0^ts^{(\al_1\tau-1)p'}\,\rd s=0$, indicating
\eqref{eq-asymp-w-init} immediately.
\end{proof}

As a direct application of Theorems \ref{thm-reg-v}--\ref{thm-reg-w}, it is
straightforward to show the Lipschitz stability of the solution with respect to
various coefficients.

\begin{proof}[Proof of Theorem \ref{thm-Lip-coef}]
Let $\ga,\tau\in(0,1]$, $a\in\cD((-L)^\ga)$ and $C>0$ be a general constant
which depends only on $a$, $\cA$, $\cQ$ and $\cU$. First, a direct application
of Theorem \ref{thm-reg-v} immediately yields $u\in L^p(0,T;H^2(\Om)\cap
H^1_0(\Om))$ and $\pa_t^\be u\in L^p(0,T;L^2(\Om))$ for $0<\be\le\al_1$, where
we abbreviate $p:=\f1{1-\ga}$. More precisely, there exists $C>0$ such that
\begin{equation}\label{eq-est-u1}
\|u\|_{L^p(0,T;H^2(\Om))}\le C,\quad\|\pa_t^\be u\|_{L^p(0,T;L^2(\Om))}\le C\
(0<\be\le\al_1).
\end{equation}
On the other hand, by taking the difference of systems \eqref{eq-gov-u2} and
\eqref{eq-gov-u1}, it turns out that the system for $v:=\wt u-u$ reads
\[
\begin{cases}
\!\begin{alignedat}{2}
& \sum_{j=1}^m\wt q_j\pa_t^{\wt\al_j}v=L_{\wt D}v+F & \quad & \mbox{in
}\Om\times(0,T],\\
& v=0 & \quad & \mbox{on }\pa\Om\times(0,T],\\
& v|_{t=0}=0 & \quad & \mbox{in }\Om,
\end{alignedat}
\end{cases}
\]
where
\[
F:=\sum_{j=1}^m\wt q_j(\pa_t^{\al_j}u-\pa_t^{\wt\al_j}u)+\sum_{j=2}^m(q_j-\wt
q_j)\pa_t^{\al_j}u+L_{\wt D-D}u.
\]
Without loss of generality, we assume $\al_1\ge\wt\al_1$, or otherwise we
investigate $v:=u-\wt u$ instead. Therefore, together with $D,\wt D\in
C^1(\ov\Om)$, we see $F\in L^p(0,T;L^2(\Om))$ from \eqref{eq-est-u1}. Now it is
straightforward to employ Theorem \ref{thm-reg-w}(b) to obtain
\begin{equation}\label{eq-est-u-F}
\|u-\wt
u\|_{L^p(0,T;\cD((-L)^{1-\tau}))}=\|v\|_{L^p(0,T;\cD((-L)^{1-\tau}))}\le\f
C\tau\|F\|_{L^p(0,T;L^2(\Om))}.
\end{equation}
Especially, if $\ga\ge\f12$, we see $p=\f1{1-\ga}\ge2$ and hence
$L^p(\Om\times(0,T))\subset L^2(0,T;L^2(\Om))$. It then follows from Theorem
\ref{thm-reg-w}(a) that
\begin{equation}\label{eq-est-u-Fsp}
\|u-\wt u\|_{L^2(0,T;H^2(\Om))}\le C\|F\|_{L^2(\Om\times(0,T))}\le
C\|F\|_{L^p(0,T;L^2(\Om))}.
\end{equation}
Therefore, it suffices to dominate $\|F\|_{L^p(0,T;L^2(\Om))}$ by the
difference of coefficients.

To this end, first it is readily seen from \eqref{eq-est-u1} that
\begin{align*}
\|F\|_{L^p(0,T;L^2(\Om))} & \le\sum_{j=1}^m\wt
q_j\|\pa_t^{\al_j}u-\pa_t^{\wt\al_j}u\|_{L^p(0,T;L^2(\Om))}+\sum_{j=2}^m|q_j-\wt
q_j|\|\pa_t^{\al_j}u\|_{L^p(0,T;L^2(\Om))}\\
& \quad\,+C\|D-\wt D\|_{C^1(\ov\Om)}\|u\|_{L^p(0,T;H^2(\Om))}\\
& \le
C\left(\sum_{j=1}^m\|\pa_t^{\al_j}u-\pa_t^{\wt\al_j}u\|_{L^p(0,T;L^2(\Om))}+\sum_{j=2}^m|q_j-\wt
q_j|+\|D-\wt D\|_{C^1(\ov\Om)}\right).
\end{align*}
To give an estimate for $\pa_t^{\al_j}u-\pa_t^{\wt\al_j}u$ by
$|\al_j-\wt\al_j|$, we adopt a similar treatment as that in
\cite[Proposition 1]{LZJY13} and decompose it by definition as
\begin{align*}
\pa_t^{\al_j}u(\,\cdot\,,t)-\pa_t^{\wt\al_j}u(\,\cdot\,,t) &
=\f1{\Ga(1-\al_j)}\int_0^t\f{\pa_su(\,\cdot\,,s)}{(t-s)^{\al_j}}\,\rd
s-\f1{\Ga(1-\wt\al_j)}\int_0^t\f{\pa_su(\,\cdot\,,s)}{(t-s)^{\wt\al_j}}\,\rd
s\\
& =I_j^1(\,\cdot\,,t)+I_j^2(\,\cdot\,,t),
\end{align*}
where
\begin{align*}
I_j^1(\,\cdot\,,t) &
:=\f{\Ga(1-\wt\al_j)-\Ga(1-\al_j)}{\Ga(1-\wt\al_j)}\pa_t^{\al_j}u(\,\cdot\,,t),\\
I_j^2(\,\cdot\,,t) &
:=\f1{\Ga(1-\wt\al_j)}\int_0^t\{(t-s)^{-\al_j}-(t-s)^{-\wt\al_j}\}\pa_su(\,\cdot\,,s)\,\rd
s.
\end{align*}
Since $\al_j,\wt\al_j\in[\underline\al,\ov\al]$ and the Gamma function is
Lipschitz continuous in $[1-\ov\al,1-\underline\al]$, it follows from
\eqref{eq-est-u1} that
\begin{equation}\label{eq-est-I1}
\|I_j^1\|_{L^p(0,T;L^2(\Om))}=\f{|\Ga(1-\wt\al_j)-\Ga(1-\al_j)|}{\Ga(1-\wt\al_j)}\|\pa_t^{\al_j}u\|_{L^p(0,T;L^2(\Om))}\le
C|\al_j-\wt\al_j|.
\end{equation}
In order to treat $I_j^2$, we recall the estimate \eqref{eq-est-v_t} for
$\pa_tu$ and utilize Minkowski's inequality for integrals to deduce
\begin{align*}
\|I_j^2(\,\cdot\,,t)\|_{L^2(\Om)} &
=\f1{\Ga(1-\wt\al_j)}\left\|\int_0^t\{(t-s)^{-\al_j}-(t-s)^{-\wt\al_j}\}\pa_su(\,\cdot\,,s)\,\rd
s\right\|_{L^2(\Om)}\\
&
\le\int_0^t|(t-s)^{-\al_j}-(t-s)^{-\wt\al_j}|\|\pa_su(\,\cdot\,,s)\|_{L^2(\Om)}\,\rd
s\\
& \le C\int_0^t|(t-s)^{-\al_j}-(t-s)^{-\wt\al_j}|\,s^{\al_1\ga-1}\,\rd s\\
& =C\int_0^t|s^{-\al_j}-s^{-\wt\al_j}|\,(t-s)^{\al_1\ga-1}\,\rd s.
\end{align*}
Using the mean value theorem, we have
\[
|s^{-\al_j}-s^{-\wt\al_j}|=|\ln s|\,s^{-\wh\al_j(s)}|\al_j-\wt\al_j|,
\]
where $\wh\al_j(s)$ is a parameter depending on $s$ such that
\[
\min\{\al_j,\wt\al_j\}\le\wh\al_j(s)\le\max\{\al_j,\wt\al_j\}\le\al_1
\]
by the assumption $\al_1\ge\wt\al_1$. Henceforth, we assume $T>1$ without lose
of generality. We prove separately in the cases $0<t\le1$ and $1<t\le T$.
First, let $0<t\le1$. Then there holds $0<s<1$ and hence
\[
s^{-\wh\al_j(s)}\le s^{-\al_1}=s^\ve s^{-\al_1-\ve},
\]
where $\ve>0$ is sufficiently small such that $\al_1(1-\ga)+\ve<1-\ga$. Since
$|\ln s|\,s^\ve\le C$ for $0<s<1$, we obtain
\begin{align}
\|I_j^2(\,\cdot\,,t)\|_{L^2(\Om)} & \le C|\al_j-\wt\al_j|\int_0^t|\ln
s|\,s^{-\wh\al_j(s)}(t-s)^{\al_1\ga-1}\,\rd s\nonumber\\
& \le C|\al_j-\wt\al_j|\int_0^t(|\ln
s|\,s^\ve)s^{-\al_1-\ve}(t-s)^{\al_1\ga-1}\,\rd s\nonumber\\
& \le C|\al_j-\wt\al_j|\,t^{-\al_1(1-\ga)-\ve},\quad0<t\le1,\label{eq-est-I2-s}
\end{align}
where we apply the boundedness of the Beta function $B(1-\al_1-\ve,\al_1\ga)$
and $\ga>0$. Second, let $1<t\le T$. Then it is readily seen that $t-s>1-s$ for
$0<s<1$ and $|\ln s|\,s^{-\wh\al_j(s)}\le C$ for $1\le s<t$. These observation,
together with the inequality \eqref{eq-est-I2-s} for $t=1$, indicate
\begin{align}
\|I_j^2(\,\cdot\,,t)\|_{L^2(\Om)} & \le
C|\al_j-\wt\al_j|\left(\int^1_0+\int_1^t\right)|\ln
s|\,s^{-\wh\al_j(s)}(t-s)^{\al_1\ga-1}\,\rd s\nonumber\\
& \le C|\al_j-\wt\al_j|\left(\int^1_0|\ln
s|\,s^{-\wh\al_j(s)}(1-s)^{\al_1\ga-1}\,\rd s+C\int_1^t(t-s)^{\al_1\ga-1}\,\rd
s\right)\nonumber\\
& \le C|\al_j-\wt\al_j|,\quad1<t\le T.\label{eq-est-I2-l}
\end{align}
The combination of \eqref{eq-est-I2-s} and \eqref{eq-est-I2-l} immediately
yields
\[
\|I_j^2(\,\cdot\,,t)\|_{L^2(\Om)}\le
C|\al_j-\wt\al_j|\,t^{-\al_1(1-\ga)-\ve},\quad0<t\le T
\]
and thus $\|I_j^2\|_{L^p(0,T;L^2(\Om))}\le C|\al_j-\wt\al_j|$ because $\al_1(1-\ga)+\ve<1-\ga=1/p$. Consequently, collecting the estimate \eqref{eq-est-I1} for $I_j^1$, we conclude
\begin{align*}
\|F\|_{L^p(0,T;L^2(\Om))} & \le
C\left(\sum_{j=1}^m\|I_j^1+I_j^2\|_{L^p(0,T;L^2(\Om))}+\sum_{j=2}^m|q_j-\wt
q_j|+\|D-\wt D\|_{C^1(\ov\Om)}\right)\\
& \le C\left(\sum_{j=1}^m|\al_j-\wt\al_j|+\sum_{j=2}^m|q_j-\wt q_j|+\|D-\wt
D\|_{C^1(\ov\Om)}\right),
\end{align*}
implying \eqref{eq-Lip-coef} and \eqref{eq-Lip-coef-sp} with the aid of
\eqref{eq-est-u-F} and \eqref{eq-est-u-Fsp} respectively.
\end{proof}

\Subsection{Proof of Theorem \ref{thm-asymp}}\label{sec-asymp}

In this subsection, we study the long-time asymptotic behavior of the solution
$u$ to \eqref{eq-gov-u}--\eqref{eq-IC-u} with $F=0$ by a Laplace transform
argument. In the sequel, by $C$ we refer to a generic constant which is
independent of the initial value $a$ and $u$ but may depend on $d$, $\Om$,
$\al_j$, $q_j$ ($j=1,\ldots,m$) and the operator $-L$.

Although an explicit representation \eqref{eq-rep-v} is available in this case,
we write the solution in form of
\begin{equation}\label{eq-rep-u-infty}
u(\,\cdot\,,t)=\sum_{n=1}^\infty u_n(t)\vp_n,\quad t>0
\end{equation}
by use of the eigensystem $\{\la_n,\vp_n\}$ of $-L$, where a direct
calculation and the orthogonality of $\{\vp_n\}$ yield
\begin{equation}\label{eq-eigen-exp}
\begin{cases}
\!\begin{alignedat}{2}
& \sum_{j=1}^mq_j\pa_t^{\al_j}u_n(t)+\la_n u_n(t)=0, & \quad & t>0,\\
& u_n(0)=(a,\vp_n),
\end{alignedat}
\end{cases}n=1,2,\ldots.
\end{equation}
The proof of Theorem \ref{thm-asymp} relies on the following lemma.

\begin{lem}\label{lem-asymp-u_n}
Let $u_n\ (n=1,2,\ldots)$ solve the initial value problem
\eqref{eq-eigen-exp}. Then there
exists a constant $C>0$ such that
\begin{equation}\label{eq-asymp-u_n}
\left|u_n(t)-\f{q_m(a,\vp_n)}{\la_n\Ga(1-\al_m)\,t^{\al_m}}\right|\le\f{C|(a,\vp_n)|}{\la_nt^{\al_{m-1}}},\quad
t\gg1.
\end{equation}
\end{lem}

\begin{proof}
We abbreviate $a_n:=(a,\vp_n)$ for simplicity. Applying the Laplace transform
to \eqref{eq-eigen-exp} and using the formula
\[\cL(\pa_t^\al f)(s)=s^\al\cL(f)(s)-s^{\al-1}f(0+),\]
we are led to the transformed algebraic equation
\[
\cL(u_n)(s)=\f{a_n}{w(s)}\sum_{j=1}^mq_j s^{\al_j-1},\quad
w(s):=\sum_{j=1}^m q_j s^{\al_j}+\la_n.
\]
Noting that the Laplace transform of $u_n$ has a branch point zero, we should
cut off the negative part of the real axis so that the function $w(s)$ has no
zero in the main sheet of the Riemann surface including its boundaries on the
cut. In fact, for $s=r\,\e^{\ri\te}$, we see that $\sin(\al_j\te)$
($j=1,\cdots,m$) have the same signal and thus
$\mathrm{Im}(w(s))=\sum_{j=1}^mq_jr^{\al_j}\sin(\al_j\te)\ne0$ since $q_j>0$.
Therefore, the inverse Laplace transform of $\cL(u_n)$ can be represented by an
integral on the Hankel path $\mathrm{Ha}(0+)$ (i.e., the loop constituted by a
small circle $|s|=\ve$ with $\ve\to0$ and by the two borders of the cut
negative real axis). Actually, it suffices to consider the following integral
\begin{equation}\label{eq-Hankel}
\f1{2\pi\,\ri}\int_\cC\cL(u_n)(s)\,\e^{st}\,\rd s
\end{equation}
and estimate each
\[
H_\ell(t;R):=\int_{\cC_\ell}\cL(u_n)(s)\,\e^{st}\,\rd s,\quad\ell=1,\cdots,5,
\]
where the loop $\cC$ and its partitions $\cC_\ell$ ($\ell=1,\ldots,5$) are
illustrated in Figure \ref{fig-loop}.
\begin{figure}[htbp]\centering
\input{fig2.tex}\\[3mm]
\caption{The loop $\cC$ and its partition.}\label{fig-loop}
\end{figure}

For $H_1(t;R)$, noting that $|s|=R>1$ and using a change of variable, we have
\begin{align*}
|H_1(t;R)| & =\left|\int_{\cC_1}\cL(u_n)(s)\,\e^{st}\,\rd s\right|\le
C|a_n|\int_{\pi/2}^\pi R^{\al_m}\,\e^{Rt\cos{\te}}\,\rd\te\\
& =C|a_n|R^{\al_m}\int_{-1}^0\f{\e^{Rt\eta}}{\sqrt{1-\eta^2}}\,\rd\eta,\quad
t>0.
\end{align*}
Furthermore, we break up the above integral in $[-1,0]$ into two parts and
calculate their bounds respectively as
\begin{align*}
R^{\al_m}\int_{-1}^0\f{\e^{Rt\eta}}{\sqrt{1-\eta^2}}\,\rd\eta &
=R^{\al_m}\left(\int_{-1}^{-1/2}+\int_{-1/2}^0\right)\f{\e^{Rt\eta}}{\sqrt{1-\eta^2}}\,\rd\eta\\
& \le
R^{\al_m}\,\e^{-Rt/2}\int_{-1}^{-1/2}\f{\rd\eta}{\sqrt{1-\eta^2}}+CR^{\al_m}\int_{-1/2}^0\e^{Rt\eta}\,\rd\eta\\
& \le CR^{\al_m}\,\e^{-Rt/2}+CR^{\al_m-1}\f{1-\e^{-Rt/2}}t\to0\quad\mbox{as
}R\to\infty,\ t>0.
\end{align*}
Therefore, for any $t>0$, we see that $H_1(t;R)\to0$ as $R\to\infty$. Similarly
to the calculation of $H_1(t;R)$, we have $H_3(t;R)\to0$ as $R\to\infty$ for
any $t>0$. On the other hand, since $R\cos\te\le\ga$ for all
$\te\in[\te_R,\pi/2]$ where $\te_R$ denotes the argument of point $A$, we have
\begin{align*}
|H_2(t;R)| & \le C|a_n|R^{\al_m}\int_{\te_R}^{\pi/2}\e^{Rt\cos\te}\,\rd\te\le
C|a_n|R^{\al_m}\,\e^{\ga t}\left(\f\pi2-\te_R\right)\\
& =C|a_n|R^{\al_m}\,\e^{\ga t}\left(\f\pi2-\arccos\f1R\right)\to0\quad\mbox{as
}R\to\infty.
\end{align*}
Therefore, since $w(s)$ has no zero in the main sheet of the Riemann surface
including the boundaries on the cut, the integral in \eqref{eq-Hankel}
vanishes. By Fourier-Mellin formula (see e.g. \cite{S91}), we have
\[
u_n(t)=\lim_{M\to\infty}\f1{2\pi\,\ri}\int_{\ga-\ri M}^{\ga+\ri
M}\cL(u_n)(s)\,\e^{st}\,\rd
s=\f1{2\pi\,\ri}\int_{\mathrm{Ha}(\ve)}\cL(u_n)(s)\,\e^{st}\,\rd s.
\]
Here the integral is taken on the segment from $\ga-\ri M$ to $\ga+\ri M$, and $\mathrm{Ha}(\ve)$ denotes the Hankel path in $\BC$ defined as
\[\mathrm{Ha}(\ve):=\{s\in\BC;\thinspace\arg\,s=\pm\pi,\
|s|\ge\ve\}\cup\{s\in\BC;\thinspace-\pi\le\arg\,s\le\pi,\ |s|=\ve\}.\]
By a similar argument as above, we find
\[
\f1{\Ga(1-\al_m)\,t^{\al_m}}=\lim_{M\to\infty}\f1{2\pi\,\ri}\int_{\ga-\ri
M}^{\ga+\ri M}s^{\al_m-1}\,\e^{st}\,\rd
s=\f1{2\pi\,\ri}\int_{\mathrm{Ha}(\ve)}s^{\al_m-1}\,\e^{st}\,\rd s.
\]

It is now straightforward to show that the contribution from the Hankel path
$\mathrm{Ha}(\ve)$ as $\ve\to0$ is provided by
\begin{align}
& u_n(t)-\f{q_ma_n}{\la_n\Ga(1-\al_m)\,t^{\al_m}}=a_n\int_0^\infty
H(r,\la_n)\,\e^{-rt}\,\rd r,\quad\mbox{where}\label{eq-Hankel-rep-u_n}\\
& H(r,\la_n):=-\f1\pi\,\mathrm{Im}\left\{\left.\left(\f1{w(s)}\sum_{j=1}^mq_js^{\al_j-1}-\f{q_m}{\la_n}s^{\al_m-1}\right)\right|_{s=r\,\e^{\ri\pi}}\right\}.\nonumber
\end{align}
To give the desired estimate \eqref{eq-asymp-u_n}, we observe that $|w(s)|\ge
C\la_n$ as long as $r=|s|\le\ve_0\la_n$, where $\ve_0>0$ is sufficiently small.
This indicates
\begin{align*}
|H(r,\la_n)| &
\le\left|\f{\la_n\sum_{j=1}^{m-1}q_js^{\al_j-1}-\sum_{j=1}^mq_jq_ms^{\al_j+\al_m-1}}{\la_n(\sum_{j=1}^mq_js^{\al_j}+\la_n)}\right|\\
& \le\f{C|a_n|}{\la_n}\left(\sum_{j=1}^{m-1}|s|^{\al_j-1}+\sum_{j=1}^m|s|^{\al_j+\al_m-1}\right),\quad\forall\,|s|\le\ve_0\la_n.
\end{align*}
Meanwhile, for any $s=r\,\e^{\pm\ri\pi}$ with $r\ge\ve_0\la_n$, we know that
\[
|H(r,\la_n)|\le\f{\sum_{j=1}^{m-1}q_jr^{\al_j-1}}{|\mathrm{Im}\sum_{j=1}^mq_js^{\al_j}|}+\f{\sum_{j=1}^mq_jq_mr^{\al_j+\al_m-1}}{\la_n|\mathrm{Im}\sum_{j=1}^mq_js^{\al_j}|}\le
C.
\]
Using these estimates, we break up the integral in \eqref{eq-Hankel-rep-u_n}
into two parts and give respective bounds as
\begin{align*}
\left|\int_0^{\ve_0\la_n}H(r,\la_n)\,\e^{-rt}\,\rd r\right| & \le\f
C{\la_n}\int_0^\infty\left(\sum_{j=1}^{m-1}r^{\al_j-1}+\sum_{j=1}^mr^{\al_j+\al_m-1}\right)\e^{-rt}\,\rd
r\\
& \le\f
C{\la_n}\left(\sum_{j=1}^{m-1}\f1{t^{\al_j}}+\sum_{j=1}^m\f1{t^{\al_j+\al_m}}\right),\\
\left|\int_{\ve_0\la_n}^\infty H(r,\la_n)\,\e^{-rt}\,\rd r\right| & \le
C\int_{\ve_0\la_n}^\infty\e^{-rt}\,\rd r=\f C{t\,\e^{\ve_0\la_nt}}\le\f C{\la_nt^2}.
\end{align*}
Collecting the above two estimates, we obtain \eqref{eq-asymp-u_n} for
sufficiently large $t$.
\end{proof}

\begin{proof}[Proof of Theorem \ref{thm-asymp}]
Let $u$ take the form of \eqref{eq-rep-u-infty} which solves
\eqref{eq-gov-u}--\eqref{eq-IC-u} with $a\in L^2(\Om)$ and $F=0$, and fix any $T>0$ sufficiently
large. For all $t\ge T$, it immediately follows from Lemma \ref{lem-asymp-u_n}
and the eigenfunction expansion that
\begin{align*}
\left\|u(\,\cdot\,,t)-\f{(-L)^{-1}(q_ma)}{\Ga(1-\al_m)\,t^{\al_m}}\right\|_{H^2(\Om)}
& \le
C\left\|\sum_{n=1}^\infty\left(u_n(t)-\f{q_m(a,\vp_n)}{\la_n\Ga(1-\al_m)\,t^{\al_m}}\right)\vp_n\right\|_{\cD(-L)}\\
&
=C\left(\sum_{n=1}^\infty\left|\la_nu_n(t)-\f{q_m(a,\vp_n)}{\Ga(1-\al_m)\,t^{\al_m}}\right|^2\right)^{1/2}\\
&
\le\f C{t^{\al_{m-1}}}\left(\sum_{n=1}^\infty|(a,\vp_n)|^2\right)^{1/2}=\f{C\|a\|_{L^2(\Om)}}{t^{\al_{m-1}}},
\end{align*}
implying $u\in C([T,\infty);H^2(\Om)\cap H^1_0(\Om))$. On the other hand, since
Theorem \ref{thm-reg-v}(a) guarantees $u\in C([0,T];L^2(\Om))\cap
C((0,T];H^2(\Om)\cap H^1_0(\Om))$, the proof is finished by combining the
regularity results in the finite and infinite time spans.
\end{proof}

\begin{rem}\label{rem-asymp}
If some $q_{j_0}$ is negative, then we cannot obtain the asymptotic estimate
for the solution $u$ of the initial-boundary value problem
\eqref{eq-gov-u}--\eqref{eq-IC-u}. In fact, for some $n\in\BN$ sufficiently
large, we study the following problem
\[
\begin{cases}
\pa_t^{1/2}u(x,t)-3\la_n\pa_t^{1/4}u(x,t)=Lu(x,t), & x\in\Om,\ t>0,\\
u(x,t)=0, & x\in\pa\Om,\ t>0.\\
u(x,0)=a_n\vp_n=(a,\vp_n)\vp_n, & x\in\Om,
\end{cases}
\]
where $(\la_n,\vp_n)$ is the $n$-th pair in the eigensystem of the elliptic
operator $-L$, and $a\in L^2(\Om)$. The Laplace transform of the solution reads
\[
\cL(u)(s)=\f{a_n}{w(s)}\left(s^{-1/2}-3\la_ns^{-3/4}\right)\vp_n,\quad
w(s):=s^{1/2}-3\la_ns^{1/4}+\la_n.
\]
We see that $\{s;\thinspace w(s)=0\}$ is a finite set with all of the zero
points having finite multiplicity in the main sheet of the Riemann surface, and
there is no zero point on the negative part of the real axis since $\la_n$ is
sufficiently large. Furthermore, we can prove that there exist zeros of $w(s)$
having positive real parts. In fact, obviously
\[
r_\pm:=\f{3\la_n\pm\sqrt{9\la_n^2-4\la_n}}2>0
\]
solves $w(r_\pm)=0$, and we have
\[
w'(r_\pm)=\f12r_\pm^{-1/2}-\f{3\la_n}4r_\pm^{-3/4}\ne0.
\]
Note that
\[
\f1{2\pi\,\ri}\int_\cC\cL(u)(s)\,\e^{st}\,\rd
s=\sum\mathrm{Res}\{\cL(u)(s)\,\e^{st},\cC\},
\]
where $\cC$ is defined in Figure \ref{fig-loop}, $\mathrm{Res}\{f,\cC\}$
denotes the residue of function $f$ in the domain enclosed by $\cC$, and the
sum is taken over all the poles of $\cL(u)(s)\,\e^{st}$ in this domain.
Repeating the argument in the proof of Lemma \ref{lem-asymp-u_n}, we deduce
\[
u(t)=\lim_{M\to\infty}\f1{2\pi\,\ri}\int_{\ga-\ri M}^{\ga+\ri
M}\cL(u)(s)\,\e^{st}\rd
s=\sum\mathrm{Res}\{\cL(u)(s)\,\e^{st}\}+\f1{2\pi\,\ri}\int_{\mathrm{Ha}(0+)}\cL(u)(s)\,\e^{st}\,\rd
s.
\]
Here the sum is taken over all the poles of $\cL(u)(s)\,\e^{st}$ lying on the
left-hand side of the line $\{z=\ga+\ri\,M;\thinspace M\in\BR\}$ with
$\ga>r_+$, and there are only finite terms in this summation since $w(s)$ only
has finite number of zero points including multiplicity in the main sheet of
the Riemann surface cutting of the negative axis. We can easily see that
\[
\mathrm{Res}\{(\cL(u)(s)\,\e^{st})|_{s=r_\pm}\}=\f{r_\pm^{-1/2}-3\la_nr_\pm^{-3/4}}{w'(r_\pm)}\,\e^{r_\pm t}\vp_n.
\]
Of course $\e^{r_\pm t}$ tend to infinity as $t\to\infty$ since $r_\pm>0$,
indicating that the asymptotic behavior in Theorem \ref{thm-asymp} does not
hold for this case.
\end{rem}

\Section{Concluding Remarks}\label{sec-rem}

We summarize this paper by providing several concluding remarks. Concerning the
initial-boundary value problem \eqref{eq-gov-u}--\eqref{eq-IC-u} for multi-term
time-fractional diffusion equations, we mainly investigate the well-posedness
and the long-time asymptotic behavior of the solution, which turn out to be
mostly parallel to those of the single-term prototype. On the basis of the
representation of solutions and a careful analysis of multinomial
Mittag-Leffler functions, we succeed in dominating the solutions by the initial
value $a$ and the source term $F$. Although uniqueness and stability also
follow from the maximum principle developed in \cite{L11}, we carry out various
estimates so that regularity and short-time asymptotic behaviors of the
solutions are directly connected with the regularity of $a$ and $F$ (see
Theorems \ref{thm-reg-v}--\ref{thm-reg-w}). Furthermore, in Theorem
\ref{thm-Lip-coef} we establish the Lipschitz stability of the solution with
respect to $\al_j$, $q_j$ and the diffusion coefficient, which is not only
important by itself but also applicable to the corresponding inverse
coefficient problem when treated by a minimization approach (see \cite[Theorem
5]{LZJY13}).

Simultaneously, we also obtain an extended version of \cite[Corollary
2.6]{SY11} in Theorem \ref{thm-asymp}, which asserts that, if the solution does
not vanish identically, then its decay rate cannot exceed $t^{-\al_m}$, where
$\al_m$ is the minimum order of fractional time-derivative. It is a remarkable
property of fractional diffusion equations because the classical diffusion
equation admits non-zero solutions decaying exponentially. This characterizes
the slow diffusion in contrast to the classical one.

In the formulation of the initial-boundary value problem, we emphasize that the
coefficients $q_j$ of the time derivatives are positive constants because this
assumption is obligatory not only to acquire explicit solutions but also to
apply the Laplace transform in time, which are essential in the discussions of
well-posedness and asymptotic behavior, respectively. On the other hand, if
$q_j$ are space-dependent, then explicit solutions are not available so that
one should rely on a fixed point argument for the unique existence of solution,
and the improvement of regularity in space is strictly less than $2$ orders
(see \cite[Theorem 2]{BY13}). On the other hand, if some $q_{j_0}$ is negative,
then one may construct a counterexample in which the asymptotic property fails
(see Remark \ref{rem-asymp}).

However, in view of practical applications and theoretical interests, the
linear non-symmetric diffusion equation with positive variable coefficients of
Caputo derivatives in time can be regarded as a more feasible model equation
than that we have studied in the current paper, but it will be definitely more
challenging. Though still under consideration, we expect to establish parallel
results for this more generalized case.

\bigskip

{\bf Acknowledgement}\ \ The second author appreciates the invaluable
discussions with Professor William Rundell, Professor Raytcho Lazarov, Dr.
Lihua Zuo and Mr. Zhi Zhou (Texas A\&M University).

\end{document}

%% file: fig1.tex
{\unitlength 0.1in
\begin{picture}( 32.2700, 33.6500)(  2.0000,-34.0000)
%
{\color[named]{Black}{%
\special{pn 8}%
\special{pa 200 1800}%
\special{pa 1200 1800}%
\special{fp}%
}}%
%
{\color[named]{Black}{%
\special{pn 20}%
\special{pa 1200 1800}%
\special{pa 1810 1800}%
\special{fp}%
}}%
%
{\color[named]{Black}{%
\special{pn 8}%
\special{pa 1800 1800}%
\special{pa 3400 1800}%
\special{fp}%
\special{sh 1}%
\special{pa 3400 1800}%
\special{pa 3334 1780}%
\special{pa 3348 1800}%
\special{pa 3334 1820}%
\special{pa 3400 1800}%
\special{fp}%
\special{pa 1800 3400}%
\special{pa 1800 200}%
\special{fp}%
\special{sh 1}%
\special{pa 1800 200}%
\special{pa 1780 268}%
\special{pa 1800 254}%
\special{pa 1820 268}%
\special{pa 1800 200}%
\special{fp}%
}}%
%
{\color[named]{Black}{%
\special{pn 8}%
\special{pa 1800 1800}%
\special{pa 1200 1000}%
\special{dt 0.045}%
\special{pa 1800 1800}%
\special{pa 1200 2600}%
\special{dt 0.045}%
}}%
%
{\color[named]{Black}{%
\special{pn 13}%
\special{pa 1200 1000}%
\special{pa 600 200}%
\special{fp}%
\special{pa 1200 2600}%
\special{pa 600 3400}%
\special{fp}%
}}%
%
{\color[named]{Black}{%
\special{pn 13}%
\special{ar 1800 1800 1000 1000  4.0688879  2.2142974}%
}}%
%
{\color[named]{Black}{%
\special{pn 8}%
\special{pn 8}%
\special{pa 1200 2600}%
\special{pa 1193 2595}%
\special{fp}%
\special{pa 1164 2572}%
\special{pa 1159 2567}%
\special{fp}%
\special{pa 1130 2542}%
\special{pa 1125 2537}%
\special{fp}%
\special{pa 1097 2511}%
\special{pa 1092 2506}%
\special{fp}%
\special{pa 1065 2479}%
\special{pa 1060 2473}%
\special{fp}%
\special{pa 1036 2445}%
\special{pa 1031 2439}%
\special{fp}%
\special{pa 1008 2410}%
\special{pa 1003 2404}%
\special{fp}%
\special{pa 982 2374}%
\special{pa 976 2367}%
\special{fp}%
\special{pa 956 2336}%
\special{pa 952 2330}%
\special{fp}%
\special{pa 934 2298}%
\special{pa 929 2292}%
\special{fp}%
\special{pa 912 2259}%
\special{pa 909 2252}%
\special{fp}%
\special{pa 892 2219}%
\special{pa 889 2212}%
\special{fp}%
\special{pa 873 2178}%
\special{pa 871 2170}%
\special{fp}%
\special{pa 859 2135}%
\special{pa 856 2128}%
\special{fp}%
\special{pa 845 2092}%
\special{pa 842 2085}%
\special{fp}%
\special{pa 832 2049}%
\special{pa 830 2041}%
\special{fp}%
\special{pa 822 2005}%
\special{pa 820 1997}%
\special{fp}%
\special{pa 814 1961}%
\special{pa 812 1953}%
\special{fp}%
\special{pa 808 1917}%
\special{pa 806 1909}%
\special{fp}%
\special{pa 804 1872}%
\special{pa 802 1864}%
\special{fp}%
\special{pa 800 1827}%
\special{pa 800 1819}%
\special{fp}%
\special{pa 800 1781}%
\special{pa 800 1773}%
\special{fp}%
\special{pa 802 1736}%
\special{pa 804 1728}%
\special{fp}%
\special{pa 806 1691}%
\special{pa 808 1684}%
\special{fp}%
\special{pa 812 1647}%
\special{pa 814 1639}%
\special{fp}%
\special{pa 820 1603}%
\special{pa 822 1595}%
\special{fp}%
\special{pa 830 1559}%
\special{pa 832 1552}%
\special{fp}%
\special{pa 842 1516}%
\special{pa 845 1508}%
\special{fp}%
\special{pa 856 1473}%
\special{pa 859 1466}%
\special{fp}%
\special{pa 872 1431}%
\special{pa 875 1423}%
\special{fp}%
\special{pa 889 1388}%
\special{pa 893 1382}%
\special{fp}%
\special{pa 909 1348}%
\special{pa 912 1341}%
\special{fp}%
\special{pa 930 1308}%
\special{pa 934 1302}%
\special{fp}%
\special{pa 953 1270}%
\special{pa 958 1264}%
\special{fp}%
\special{pa 977 1233}%
\special{pa 982 1226}%
\special{fp}%
\special{pa 1003 1196}%
\special{pa 1008 1190}%
\special{fp}%
\special{pa 1031 1161}%
\special{pa 1037 1155}%
\special{fp}%
\special{pa 1061 1127}%
\special{pa 1067 1121}%
\special{fp}%
\special{pa 1092 1094}%
\special{pa 1098 1088}%
\special{fp}%
\special{pa 1125 1063}%
\special{pa 1131 1057}%
\special{fp}%
\special{pa 1159 1033}%
\special{pa 1165 1028}%
\special{fp}%
\special{pa 1194 1005}%
\special{pa 1200 1000}%
\special{fp}%
}}%
%
{\color[named]{Black}{%
\special{pn 8}%
\special{pa 1800 1800}%
\special{pa 200 600}%
\special{fp}%
\special{pa 1800 1800}%
\special{pa 200 3000}%
\special{fp}%
}}%
%
{\color[named]{Black}{%
\special{pn 8}%
\special{pa 200 600}%
\special{pa 200 3000}%
\special{ip}%
}}%
%
{\color[named]{Black}{%
\special{pn 8}%
\special{ar 1800 1800 200 200  4.0688879  6.2831853}%
}}%
%
{\color[named]{Black}{%
\special{pn 8}%
\special{ar 1800 1800 400 400  3.7850938  6.2831853}%
}}%
%
{\color[named]{Black}{%
\special{pn 4}%
\special{pa 910 1130}%
\special{pa 240 1800}%
\special{fp}%
\special{pa 870 1110}%
\special{pa 200 1780}%
\special{fp}%
\special{pa 840 1080}%
\special{pa 200 1720}%
\special{fp}%
\special{pa 800 1060}%
\special{pa 200 1660}%
\special{fp}%
\special{pa 770 1030}%
\special{pa 200 1600}%
\special{fp}%
\special{pa 740 1000}%
\special{pa 200 1540}%
\special{fp}%
\special{pa 700 980}%
\special{pa 200 1480}%
\special{fp}%
\special{pa 670 950}%
\special{pa 200 1420}%
\special{fp}%
\special{pa 630 930}%
\special{pa 200 1360}%
\special{fp}%
\special{pa 600 900}%
\special{pa 200 1300}%
\special{fp}%
\special{pa 560 880}%
\special{pa 200 1240}%
\special{fp}%
\special{pa 530 850}%
\special{pa 200 1180}%
\special{fp}%
\special{pa 500 820}%
\special{pa 200 1120}%
\special{fp}%
\special{pa 460 800}%
\special{pa 200 1060}%
\special{fp}%
\special{pa 430 770}%
\special{pa 200 1000}%
\special{fp}%
\special{pa 390 750}%
\special{pa 200 940}%
\special{fp}%
\special{pa 360 720}%
\special{pa 200 880}%
\special{fp}%
\special{pa 320 700}%
\special{pa 200 820}%
\special{fp}%
\special{pa 290 670}%
\special{pa 200 760}%
\special{fp}%
\special{pa 260 640}%
\special{pa 200 700}%
\special{fp}%
\special{pa 220 620}%
\special{pa 200 640}%
\special{fp}%
\special{pa 940 1160}%
\special{pa 300 1800}%
\special{fp}%
\special{pa 980 1180}%
\special{pa 360 1800}%
\special{fp}%
\special{pa 950 1270}%
\special{pa 420 1800}%
\special{fp}%
\special{pa 890 1390}%
\special{pa 480 1800}%
\special{fp}%
\special{pa 850 1490}%
\special{pa 540 1800}%
\special{fp}%
\special{pa 820 1580}%
\special{pa 600 1800}%
\special{fp}%
\special{pa 810 1650}%
\special{pa 660 1800}%
\special{fp}%
\special{pa 800 1720}%
\special{pa 720 1800}%
\special{fp}%
}}%
%
{\color[named]{Black}{%
\special{pn 4}%
\special{pa 800 1840}%
\special{pa 200 2440}%
\special{fp}%
\special{pa 810 1890}%
\special{pa 200 2500}%
\special{fp}%
\special{pa 810 1950}%
\special{pa 200 2560}%
\special{fp}%
\special{pa 820 2000}%
\special{pa 200 2620}%
\special{fp}%
\special{pa 830 2050}%
\special{pa 200 2680}%
\special{fp}%
\special{pa 840 2100}%
\special{pa 200 2740}%
\special{fp}%
\special{pa 860 2140}%
\special{pa 200 2800}%
\special{fp}%
\special{pa 880 2180}%
\special{pa 200 2860}%
\special{fp}%
\special{pa 890 2230}%
\special{pa 200 2920}%
\special{fp}%
\special{pa 910 2270}%
\special{pa 200 2980}%
\special{fp}%
\special{pa 940 2300}%
\special{pa 380 2860}%
\special{fp}%
\special{pa 960 2340}%
\special{pa 620 2680}%
\special{fp}%
\special{pa 980 2380}%
\special{pa 860 2500}%
\special{fp}%
\special{pa 780 1800}%
\special{pa 200 2380}%
\special{fp}%
\special{pa 720 1800}%
\special{pa 200 2320}%
\special{fp}%
\special{pa 660 1800}%
\special{pa 200 2260}%
\special{fp}%
\special{pa 600 1800}%
\special{pa 200 2200}%
\special{fp}%
\special{pa 540 1800}%
\special{pa 200 2140}%
\special{fp}%
\special{pa 480 1800}%
\special{pa 200 2080}%
\special{fp}%
\special{pa 420 1800}%
\special{pa 200 2020}%
\special{fp}%
\special{pa 360 1800}%
\special{pa 200 1960}%
\special{fp}%
\special{pa 300 1800}%
\special{pa 200 1900}%
\special{fp}%
\special{pa 240 1800}%
\special{pa 200 1840}%
\special{fp}%
}}%
%
{\color[named]{Black}{%
\special{pn 4}%
\special{pa 1250 1790}%
\special{pa 880 1420}%
\special{fp}%
\special{pa 1310 1790}%
\special{pa 900 1380}%
\special{fp}%
\special{pa 1370 1790}%
\special{pa 920 1340}%
\special{fp}%
\special{pa 1430 1790}%
\special{pa 940 1300}%
\special{fp}%
\special{pa 1490 1790}%
\special{pa 960 1260}%
\special{fp}%
\special{pa 1550 1790}%
\special{pa 990 1230}%
\special{fp}%
\special{pa 1610 1790}%
\special{pa 1090 1270}%
\special{fp}%
\special{pa 1670 1790}%
\special{pa 1330 1450}%
\special{fp}%
\special{pa 1730 1790}%
\special{pa 1570 1630}%
\special{fp}%
\special{pa 1190 1790}%
\special{pa 860 1460}%
\special{fp}%
\special{pa 1140 1800}%
\special{pa 850 1510}%
\special{fp}%
\special{pa 1080 1800}%
\special{pa 830 1550}%
\special{fp}%
\special{pa 1020 1800}%
\special{pa 820 1600}%
\special{fp}%
\special{pa 960 1800}%
\special{pa 810 1650}%
\special{fp}%
\special{pa 900 1800}%
\special{pa 810 1710}%
\special{fp}%
\special{pa 840 1800}%
\special{pa 800 1760}%
\special{fp}%
}}%
%
{\color[named]{Black}{%
\special{pn 4}%
\special{pa 1210 2230}%
\special{pa 800 1820}%
\special{fp}%
\special{pa 1250 2210}%
\special{pa 840 1800}%
\special{fp}%
\special{pa 1280 2180}%
\special{pa 900 1800}%
\special{fp}%
\special{pa 1320 2160}%
\special{pa 960 1800}%
\special{fp}%
\special{pa 1350 2130}%
\special{pa 1020 1800}%
\special{fp}%
\special{pa 1390 2110}%
\special{pa 1080 1800}%
\special{fp}%
\special{pa 1420 2080}%
\special{pa 1140 1800}%
\special{fp}%
\special{pa 1450 2050}%
\special{pa 1210 1810}%
\special{fp}%
\special{pa 1490 2030}%
\special{pa 1270 1810}%
\special{fp}%
\special{pa 1520 2000}%
\special{pa 1330 1810}%
\special{fp}%
\special{pa 1560 1980}%
\special{pa 1390 1810}%
\special{fp}%
\special{pa 1590 1950}%
\special{pa 1450 1810}%
\special{fp}%
\special{pa 1630 1930}%
\special{pa 1510 1810}%
\special{fp}%
\special{pa 1660 1900}%
\special{pa 1570 1810}%
\special{fp}%
\special{pa 1690 1870}%
\special{pa 1630 1810}%
\special{fp}%
\special{pa 1730 1850}%
\special{pa 1690 1810}%
\special{fp}%
\special{pa 1180 2260}%
\special{pa 810 1890}%
\special{fp}%
\special{pa 1150 2290}%
\special{pa 810 1950}%
\special{fp}%
\special{pa 1110 2310}%
\special{pa 830 2030}%
\special{fp}%
\special{pa 1080 2340}%
\special{pa 850 2110}%
\special{fp}%
\special{pa 1040 2360}%
\special{pa 890 2210}%
\special{fp}%
\special{pa 1010 2390}%
\special{pa 970 2350}%
\special{fp}%
}}%
\put(35.0000,-18.0000){\makebox(0,0){$x$}}%
\put(18.0000,-1.0000){\makebox(0,0){$y$}}%
\put(31.0000,-3.0000){\makebox(0,0){$\mathbb C$}}%
\put(19.0000,-19.0000){\makebox(0,0){$0$}}%
\put(29.0000,-19.0000){\makebox(0,0){$R$}}%
\put(21.0000,-17.0000){\makebox(0,0){$\theta$}}%
\put(21.0000,-14.0000){\makebox(0,0){$\mu$}}%
\put(24.0000,-28.0000){\makebox(0,0){$\gamma(R,\theta)$}}%
\put(5.0000,-18.0000){\makebox(0,0){{\colorbox[named]{White}{$A$}}}}%
\put(5.0000,-18.0000){\makebox(0,0){{\colorbox[named]{White}{\color[named]{Black}{$A$}}}}}%
\put(11.0000,-16.0000){\makebox(0,0){{\colorbox[named]{White}{$B$}}}}%
\put(11.0000,-16.0000){\makebox(0,0){{\colorbox[named]{White}{\color[named]{Black}{$B$}}}}}%
\put(12.0000,-19.5000){\makebox(0,0){{\colorbox[named]{White}{$-K$}}}}%
\put(12.0000,-19.5000){\makebox(0,0){{\colorbox[named]{White}{\color[named]{Black}{$-K$}}}}}%
\end{picture}}%

%% file: fig2.tex
{\unitlength 0.1in
\begin{picture}( 23.1200, 25.6500)(  1.1500,-26.0000)
%
{\color[named]{Black}{%
\special{pn 8}%
\special{pa 200 1600}%
\special{pa 2400 1600}%
\special{fp}%
\special{sh 1}%
\special{pa 2400 1600}%
\special{pa 2334 1580}%
\special{pa 2348 1600}%
\special{pa 2334 1620}%
\special{pa 2400 1600}%
\special{fp}%
\special{pa 1400 2800}%
\special{pa 1400 200}%
\special{fp}%
\special{sh 1}%
\special{pa 1400 200}%
\special{pa 1380 268}%
\special{pa 1400 254}%
\special{pa 1420 268}%
\special{pa 1400 200}%
\special{fp}%
}}%
%
{\color[named]{Black}{%
\special{pn 8}%
\special{pa 2000 800}%
\special{pa 2000 2400}%
\special{fp}%
\special{pa 400 1550}%
\special{pa 1280 1550}%
\special{fp}%
\special{pa 400 1650}%
\special{pa 1280 1650}%
\special{fp}%
}}%
%
{\color[named]{Black}{%
\special{pn 8}%
\special{ar 1400 1600 130 130  3.5363838  2.7468015}%
}}%
%
{\color[named]{Black}{%
\special{pn 8}%
\special{ar 1400 1600 1000 1000  3.1920548  5.3558901}%
}}%
%
{\color[named]{Black}{%
\special{pn 8}%
\special{ar 1400 1600 1000 1000  0.9272952  3.0916343}%
}}%
\put(21.0000,-3.0000){\makebox(0,0){$\mathbb C$}}%
\put(25.0000,-16.0000){\makebox(0,0){$x$}}%
\put(14.0000,-1.0000){\makebox(0,0){$y$}}%
\put(21.0000,-7.0000){\makebox(0,0){$A$}}%
\put(21.0000,-25.0000){\makebox(0,0){$B$}}%
\put(3.0000,-15.0000){\makebox(0,0){$C$}}%
\put(3.0000,-17.0000){\makebox(0,0){$D$}}%
\put(21.0000,-17.0000){\makebox(0,0){$\gamma$}}%
\put(16.0000,-17.0000){\makebox(0,0){$\varepsilon$}}%
\put(12.0000,-5.0000){\makebox(0,0){$\mathrm iR$}}%
\put(5.0000,-9.0000){\makebox(0,0){$\mathcal C_1$}}%
\put(5.0000,-23.0000){\makebox(0,0){$\mathcal C_3$}}%
\put(18.0000,-5.5000){\makebox(0,0){$\mathcal C_2$}}%
\put(18.0000,-26.5000){\makebox(0,0){$\mathcal C_4$}}%
\put(21.0000,-21.0000){\makebox(0,0){$\mathcal C_5$}}%
\end{picture}}%